\documentclass[a4paper]{amsart}
\usepackage{a4}
\usepackage{amsmath,amssymb,mathrsfs}
\usepackage[UKenglish] {babel}
\usepackage{color}
\usepackage{tikz}

\def\a{\alpha} \def\b{\beta} \def\d{\delta} \def\e{\epsilon} \def\f{\varphi}
\def\l{\lambda} \def\s{\sigma} \def\R{\mathbb{R}} \def\C{\mathbb{C}}
\def\H{\mathbb{H}}
\def\O{\mathbb{O}} \def\I{\mathbb{I}}
\def\Z{\mathbb{Z}}
 
\def\({\left(} \def\){\right)} 
\def\<{\langle} \def\>{\rangle}

\def\inv{^{-1}}

\renewcommand\ge{\geqslant}
\renewcommand\le{\leqslant}
\newcommand\ie{i.e.}
\newcommand\cf{cf.}
\newcommand\eg{e.g.}

\newcommand\forget[1]{}

\newcommand\cc{\mathrm{C}}

\newcommand{\triv}{\mathrm{triv}}
\newcommand{\sgn}{\mathrm{sign}}
\newcommand{\nat}{\mathrm{nat}}

\DeclareMathOperator{\spann}{span}
\DeclareMathOperator{\rt}{R}
\DeclareMathOperator{\lt}{L}
\DeclareMathOperator{\sign}{sign}

\DeclareMathOperator{\End}{End}
\DeclareMathOperator{\Aut}{Aut}

\DeclareMathOperator{\og}{O}
\DeclareMathOperator{\so}{SO}
\DeclareMathOperator{\go}{GO}
\DeclareMathOperator{\Pds}{Pds}
\DeclareMathOperator{\spds}{SPds}
\DeclareMathOperator{\sym}{Sym}
\DeclareMathOperator{\gl}{GL}
\DeclareMathOperator{\Mor}{Mor}
\DeclareMathOperator{\im}{im}
\DeclareMathOperator{\SL}{SL}

\renewcommand{\Im}{\mathop{\rm Im}\nolimits}

\newcommand\Ii{\mathop{\mathscr{I}}\nolimits}

\DeclareMathOperator{\Atp}{Atp}
\DeclareMathOperator{\Tder}{Tder}
\DeclareMathOperator{\N}{N}
\DeclareMathOperator{\rC}{C}
\DeclareMathOperator{\su}{su}
\DeclareMathOperator{\Lsl}{sl}
\DeclareMathOperator{\tr}{trace}
\DeclareMathOperator{\supp}{s}
\DeclareMathOperator{\Hom}{Hom}
\newcommand{\op}{^\textrm{op}}
\newcommand{\cS}{\mathcal{S}}
\newcommand{\cZ}{\mathcal{Z}}
\newcommand\bi{{\boldsymbol{i}}}
\newcommand\bj{{\boldsymbol{j}}}

\newenvironment{smatrix}{\left(\begin{smallmatrix}}{\end{smallmatrix}\right)}
\newcommand\smatr[1]{\begin{smatrix}#1\end{smatrix}}

\newtheorem{thm}{Theorem}
\newtheorem{lma}[thm]{Lemma}
\newtheorem{prop}[thm]{Proposition}
\newtheorem{cor}[thm]{Corollary}
\theoremstyle{definition}

\theoremstyle{remark}
\newtheorem{rmk}[thm]{Remark}
\newtheorem{exa}[thm]{Example}

\title{Inversion and quasigroup identities in division algebras}
\author{Erik Darp\"{o}}
\author{Jos\'e Maria P\'erez Izquierdo}
\address{E.~Darp\"o: Mathematical Institute \\ 24-29 St Giles'\\
Oxford OX1~3LB\\ United Kingdom.}
\address{J.M.~P\'erez Izquierdo: Departamento de Matem\'aticas y Computaci\'on\\
Universidad de La Rioja\\ 26004 Logro\~no\\ Spain.}
\email{erik.darpo@maths.ox.ac.uk, jm.perez@unirioja.es}
\thanks{The first author was partly supported by the Swedish Research Council, Grant
  no.~623-2009709, the second author by the Spanish Ministerio de Ciencia e Innovaci\'on
  (MTM2010-18370-C04-03).}

\begin{document}
\selectlanguage{UKenglish}
\date{}

\begin{abstract}
The present article is concerned with division algebras that are structurally close to
alternative algebras, in the sense that they satisfy some identity or other algebraic
property that holds for all alternative division algebras.

Motivated by Belousov's ideas on quasigroups, 
we explore a new approach to the classification of division algebras.
By a detailed study of the representations of the Lie group of autotopies of real division
algebras we show that, if the group of autotopies has a sufficiently rich structure then
the algebra is isotopic to an alternative division algebra.
On the other hand, it is straightforward to check that required conditions hold for large
classes of real division algebras, including many defined by identites expressable in
a quasigroup.

Some of the algebras that appear in our results are characterized
by the existence of a well-behaved inversion map.
We give an irredundant classification of these algebras in dimension 4, and partial
results in the 8-dimensional case. 
\end{abstract}

\maketitle

\section{Introduction}
\emph{All algebras discussed in this paper are assumed to be finite dimensional.}

\medskip

A not necessarily associative algebra $(A,xy)$ over a field $k$ is a \emph{division
  algebra} if $A\ne0$ and the $k$-linear maps $\lt_a:A\to A,\; x\mapsto ax$ and
$\rt_a:A\to A,\; x\mapsto xa$ are bijective for all non-zero $a\in A$.
Another way to express this is to say $A$ is an algebra such that $A\setminus\{0\}$ is a
\emph{quasigroup}: a non-empty set with a binary operation such that left and right
multiplication with any element are bijections.

An \emph{isotopy} between two $k$-algebras
$(A,xy)$ and $(B,x*y)$ is a triple $(\f_1,\f_2,\f_3)$ of bijective $k$-linear maps $A\to
B$ such that $\f_1(xy)=\f_2(x)*\f_3(y)$ for all $x,y\in A$.   The algebra $(A,xy)$ is an
\emph{isotope} of $(B,x*y)$ if there exits an isotopy between $(A,xy)$ and
$(B,x*y)$. Isotopes of division algebras are again division algebras. 

The most well-known examples of real division algebras are  $\R$,
$\C$, Hamilton's quaternions $\H$ and Graves' and Cayley's octonions $\O$. All of them are
composition algebras and possess an identity element. A \emph{composition algebra} is an algebra $(A,xy)$ over a field $k$ of characteristic
different from two equipped with a non-degenerate quadratic form $n:A\to k$, (the \emph{norm} of $A$) satisfying $n(xy)=n(x)n(y)$ for all $x,y\in A$. A composition algebra with an identity element is called a \emph{Hurwitz algebra}. Hurwitz algebras exist in dimension 1, 2, 4 and 8 only, and they are determined up to
isomorphism by the equivalence class of their norm. Over the real numbers, the Hurwitz division algebras are precisely $\R, \C$, $\H$ and $\O$. Isotopes of these algebras are a valuable source of real division algebras.

An algebra $(A,xy)$ is \emph{alternative} if the identities
$$x^2y=x(xy) \qquad\mbox{and}\qquad yx^2=(yx)x $$
hold in $A$, equivalently, if every subalgebra of $A$ generated by two elements is
associative.

Zorn \cite{zorn31} showed in 1931 that every central simple alternative algebra (in
particular every alternative division algebra) is either associative or a so-called
octonion algebra (\ie, an eight-dimensional Hurwitz algebra).
In particular, every alternative division algebra over $\R$ is isomorphic to either $\R, \C$, $\H$ or $\O$.
Later generalisations of Zorn's result includes the classification of all real
power-associative division algebras of dimension four \cite{osborn62,zur} and all real
flexible division algebras \cite{bbo82,cdkr99,coll,nform}.

In every alternative algebra, the \emph{Moufang} and \emph{Bol identities}
\begin{align*}
  ((xy)x)z&=x(y(xz)) && \mbox{(the left Moufang identity),} \\
  z(x(yx))&=((zx)y)x && \mbox{(the right Moufang identity),} \\
  (xy)(zx)&=(x(yz))x && \mbox{(the middle Moufang identity),}\\
  (x(yx))z&=x(y(xz)) && \mbox{(the left Bol identity),} \\
  z((xy)x)&=((zx)y)x && \mbox{(the right Bol identity)}
\end{align*}
hold.  In \cite{Ku96_2} Kunen proves that a quasigroup satisfying any of the Moufang identities has unit element. Quasigroups with (resp.\ left, right) unit element are called (resp.\ left, right) \emph{loops}. In particular, a division algebra satisfying any of the Moufang identities  is alternative. In \cite{cu}, Cuenca-Mira classifies all real division algebras satisfying any of the Bol identities.
He shows that every such algebra can be obtained from an alternative
algebra $(A,xy)$ by changing the product to either $x\circ y=\s(x)y$ or $x\circ y=x\s(y)$,
where $\s$ is an involutive automorphism of $(A,xy)$. Results of this kind, where certain hypotheses imply that the algebra can be described as an isotope of another well-known algebra were very much promoted in the theory of quasigroups  by V.D. Belousov. One of the most popular theorems in this direction is Belousov's theorem about balanced identities \cite{Be66}. A \emph{balanced identity} is a non-trivial identity of the form
\begin{equation} \label{eq:neat}
p(x_1,\dots, x_n)=q(x_1,\dots,x_n),
\end{equation}
where $p$ and $q$ are two distinct, non-associative words in the letters $x_1,\ldots,x_n$,
and the degree of each letter in both $p$ and $q$ is one.
Two variables $x_i, x_j$ are said to be \emph{separated} in $q$ if neither $x_ix_j$ nor
$x_jx_i$ occur in $q$. {A generalization of Belousov's theorem, given in \cite{Ta78},
  establishes that if a quasigroup satisfies a balanced identity $p(x_1,\dots,x_n) =
  q(x_1,\dots,x_n)$ with the property that $p$ contains a subword $x_ix_j$ and $x_i$ and
  $x_j$ are separated in $q$ then the quasigroup is an isotope of a group. See
  \cite{Fa71} for another generalization and \cite{KrTa91} for a discussion on the topic.
}

In this paper we will follow Belousov's ideas  by studing some general conditions under which a real division algebra is an isotope of $\R, \C, \H$ or $\O$. We will concentrate on the
following types of division algebras:
\begin{enumerate}
\item division algebras with inversion on the left,
\item division algebras satisfying a quasigroup identity.
\end{enumerate}
Usually two extra operations $(x,y) \mapsto
x\backslash y = \lt^{-1}_x(y)$ and $(x,y) \to x/y = \rt^{-1}_y(x)$ are considered on quasigroups. They are related to the product by
\begin{displaymath}
x\backslash (x y) = y = x (x\backslash y) \quad \text{and} \quad (x y)/y = x
= (x/y) y.
\end{displaymath}
We say that a division algebra $(A,xy)$ satisfies a \emph{quasigroup identity} if the
quasigroup $A\setminus \{ 0\}$ satisfies some identity expressible only in terms of the
operations $xy$, $x\backslash y$ and $x/y$ (\ie, they do not involve addition,
substraction or multiplication by scalars). To illustrate our results, let us consider for
instance identities in three variables $x,y,z$, in which $x$ appears three times on each
side while $y,z$ only appear once. We further assume that on each side, $x,y$ and $z$
appear from left to right in the following order: $x,x,y,x$ and $z$. There are 14 ways of
placing parentheses on $xxyxz$, so there are 91 possible identities. Some of them are
immediate consequences of identities of lower degree and they will be not considered. This
is the case, for example, with $(((xx)y)x)z=((x(xy))x)z$ and $((x(xy))x)z=(x(xy))(xz)$,
and many others. The remaining identities are collected in Table \ref{Tb:d5}. A direct
application of our Corollary \ref{cor:PI} shows that, with the exception of the identities
2, 7, 10, 13, 22, 25 and 40, a real division algebra satisfying any of these
identities is isotopic to either $\R, \C, \H$ or $\O$. This approach also works when the
order of the variables is $xyxxz$, $xxyzx$, $yxxzx$, $xyzxx$ or $yxzxx$, while it is less
successful when the order is $yzxxx$, $yxxxz$ or $xxxyz$. 

\begin{table}
\caption{Some examples of quasigroup identities}
\centering\begin{tabular}{|c|c|c|c|}
    \hline
  1 & $(((xx)y)x)z=(x(xy))(xz) $ & 2 & $(((xx)y)x)z=(xx)((yx)z) $ \\
  3 & $(((xx)y)x)z=(xx)(y(xz)) $ & 4 & $ (((xx)y)x)z=x(((xy)x)z)$ \\
  5 & $(((xx)y)x)z=x((x(yx))z) $ & 6 & $ (((xx)y)x)z=x((xy)(xz))$ \\
  7 & $(((xx)y)x)z=x(x((yx)z))$ & 8 & $(((xx)y)x)z=x(x(y(xz)))$ \\
  9 & $((x(xy))x)z=((xx)y)(xz)$ & 10 & $((x(xy))x)z=(xx)((yx)z)$ \\
  11 & $((x(xy))x)z=(xx)(y(xz))$ & 12 & $((x(xy))x)z=x((x(yx))z)$ \\
  13 & $((x(xy))x)z=x(x((yx)z))$ & 14 & $((x(xy))x)z=x(x(y(xz)))$ \\
  15 & $((xx)(yx))z=((xx)y)(xz)$ & 16 & $((xx)(yx))z=(x(xy))(xz)$ \\
  17 & $((xx)(yx))z=(xx)(y(xz))$ & 18 & $((xx)(yx))z=x(((xy)x)z)$ \\
  19 & $((xx)(yx))z=x((xy)(xz))$ & 20 & $((xx)(yx))z=x(x(y(xz)))$ \\
  21 & $(x((xy)x))z=((xx)y)(xz)$ & 22 & $(x((xy)x))z=(xx)((yx)z)$ \\
  23 & $(x((xy)x))z=(xx)(y(xz))$ & 24 & $(x((xy)x))z=x((x(yx))z)$ \\
  25 & $(x((xy)x))z=x(x((yx)z))$ & 26 & $(x((xy)x))z=x(x(y(xz)))$ \\
  27 & $(x(x(yx)))z=((xx)y)(xz)$ & 28 & $(x(x(yx)))z=(x(xy))(xz)$ \\
  29 & $(x(x(yx)))z=(xx)(y(xz))$ & 30 & $(x(x(yx)))z=x(((xy)x)z)$ \\
  31 & $(x(x(yx)))z=x((xy)(xz))$ & 32 & $(x(x(yx)))z=x(x(y(xz)))$ \\
  33 & $((xx)y)(xz)=(xx)((yx)z)$ & 34 & $((xx)y)(xz)=x(((xy)x)z)$ \\
  35 & $((xx)y)(xz)=x((x(yx))z)$ & 36 & $((xx)y)(xz)=x(x((yx)z))$ \\
  37 & $(x(xy))(xz)=(xx)((yx)z)$ & 38 & $(x(xy))(xz)=x((x(yx))z)$ \\
  39 & $(x(xy))(xz)=x(x((yx)z))$ & 40 & $(xx)((yx)z)=x(((xy)x)z)$ \\
  41 & $(xx)((yx)z)=x((xy)(xz))$ & 42 & $(xx)((yx)z)=x(x(y(xz)))$ \\
  43 & $(xx)(y(xz))=x(((xy)x)z)$ & 44 & $(xx)(y(xz))=x((x(yx))z)$ \\
  45 & $(xx)(y(xz))=x(x((yx)z))$ &  &  \\
  \hline
\end{tabular}
\label{Tb:d5}
\end{table}

A division algebra $(A,xy)$ is said to have \emph{inversion on the left} (or, for short, have
\emph{inversion}) if for every non-zero
element $a\in A$ there exists an element $b\in A$ such that $\lt_a\inv = \lt_b$.
If $A$ has inversion then the element $b\in A$ is uniquely determined by $a$, and the map
$s:A\setminus\{0\}\to A\setminus\{0\}$ defined by $\lt_a\inv=\lt_{s(a)}$ is called the
\emph{inversion map} on $A$.
Division algebras with inversion are studied in Section~\ref{sec:inversion}. A complete
and irredundant classification, in 11 parameters, is given in dimension four over $\R$.
In dimension eight, the isomorphism classes of left-unital real division algebras with
inversion are shown to be parametrized by the orbits of an action of the Lie group
$\mathcal{G}_2$ on a certain, 56-dimensional manifold.
We also study division algebras $(A,xy)$ with inversion for which the inversion map $s$
satisfies $s(ab)=s(b)s(a)$ for all $a,b\in A$, and show that over the real numbers, there
are precisely ten isomorphism classes of such algebras.

In Section~\ref{sec:qgi} we state a general result (Theorem~\ref{thm:PI} and
Corollary~\ref{cor:PI}) with which one can prove that, for many types of identities, the
division algebras satisfying them are isotopes of $\R, \C, \H$ or $\O$. 
In Section~\ref{sec:examples} we use this criterion to classify certain families of
division algebras. Finally, the proof of Theorem~\ref{thm:PI} is given in
Section~\ref{sec:proof}.

Although our main focus is on real division algebras, results are stated for algebras over
general fields when possible.

A few words about notation. For any algebra $(A,xy)$, 
we write $\lt_A=\{\lt_a \mid a \in A\}$ and $\rt_A=\{\rt_a \mid a \in A\}$.
If some other algebra structure $x\circ y$ is defined on $A$, we use
$\lt^\circ_a(x)=a\circ x$ and $\rt_a^\circ(x)=x\circ a$ to denote the left and right
multiplication with respect to $\circ$, and define $\lt_A^\circ$ and $\rt_A^\circ$ analogously.
The $n$th power of an element $x$ in $(A,\circ)$ (when well defined) is written as
$x^{\circ n}$.
The inverse of a non-zero element $x$ in a Hurwitz division algebra $A$ is denoted by
$x\inv$, regardless of what symbol is used to denote the product in $A$. 

Recall that Hurwitz algebras $(A,xy)$ are \emph{quadratic algebras}: there exists a linear
form $t$ on $A$ such that
\begin{displaymath}
x^2 - t(x)x + n(x)1_A = 0
\end{displaymath}
for all $x\in A$. The linear form $t$ is called the \emph{trace} of $A$, and satisfies
$t(x) = 2(x,1_A)$, where $(x,y)=\frac{1}{2}(n(x+y)-n(x)-n(y))$ is the bilinear form associated with the quadratic form $n$. The kernel of $t$, denoted by $\Im A$, is the orthogonal complement of the identity
element with respect to $(\cdot,\cdot)$, also
$$\Im A=\{x\in A\setminus k1_A \mid x^2\in k1_A\}\cup\{0\} \,.$$
Every Hurwitz algebra $(A,xy)$ has a distinguished  anti-automorphism of
order two (or involution) $\kappa:A\to A,\:x \mapsto \bar{x} = t(x)1_A -x$ called the
\emph{standard involution}. An element $a\in A$ is invertible if and only if $n(a)\neq 0$,
in which case $a\inv=n(a)\inv\bar{a}$.
We refer to \cite{sv00} for a detailed account of the many properties of Hurwitz algebras.

\section{Division algebras with inversion on the left} \label{sec:inversion}

Given any Euclidean space $V$, denote by $\Pds(V)$ the set of positive definite symmetric
linear endomorphisms of $V$, and $\spds(V)=\Pds(V)\cap\SL(V)$.
If $(A,xy)$ is an algebra and $\a,\b\in\gl(A)$, then $A_{\a,\b}=(A,xy)_{\a,\b}$ denotes
the isotope $(A,\circ)$ of $A$ with multiplication defined by $x\circ y=\a(x)\b(y)$. 

\begin{lma} \label{hua}
  If $(A,xy)$ is a division algebra with inversion on the left, then $\lt_a\lt_b\lt_a\in\lt_A$
  for all $a,b\in A$.
\end{lma}

\begin{proof}
  The statement is obvious in case $\lt_b=\lt_a\inv$, or if either of $a,b$ is zero.
  Otherwise, Hua's identity (see \cite[p.~92]{basicI}) implies
  $$\lt_a\lt_b\lt_a =
  \lt_a+\left( \left(\lt_a -\lt_b\inv \right)\inv -\lt_a\inv \right)\inv \,.$$
  Since $A$ has inversion on the left, the right hand side of the equation is in $\lt_A$.
\end{proof}

\begin{prop} \label{general}
  \begin{enumerate}
  \item A division algebra is alternative if and only if it is unital and has inversion on
    the left.
  \item If $(A,xy)$ is a division algebra with inversion and $e\in A$ a non-zero element, then
    $A=B_{\a,\b}$ for an alternative division algebra $B=(A,\ast)$, with
    $\a=\rt_e$ and $\b=\lt_e$ being, respectively, the right and left
    multiplication operators in $A$ with the element $e$.
  \item Let $(B,xy)$ be a division algebra with inversion on the left, and $\a,\b\in\gl(B)$.
    Now $B_{\a,\b}$ has inversion on the left if and only if $\b\lt_B\b=\lt_B$ or,
    equivalently, if and only if there exists $\psi\in\gl(B)$ such that
    $\b(xy)=\psi(x)\b\inv(y)$ for all $x,y\in B$.
  \end{enumerate}
\end{prop}

\begin{proof}
1.
Every alternative division algebra $(A,xy)$ is unital \cite[Theorem~3.10]{schafer95}, and every
$a\in A\setminus\{0\}$ has a two-sided inverse $a\inv$ satisfying $\lt_a\inv=\lt_{a\inv}$.
Conversely, assume $(A,xy)$ is a unital division algebra with inversion on the left. By
Lemma~\ref{hua}, for all $a,b\in A\setminus\{0\}$ there is a $c\in A$ such that
$\lt_a\lt_b\lt_a=\lt_c$. Inserting $b=1$ and evalutating in $1$ gives
$\lt_a^2=\lt_{a^2}$ for all $a\in A$, \ie, $A$ is left alternative. But every left
alternative division algebra is alternative \cite[p.~345]{zsss}, so $A$ is alternative.

2.
Let $(A,xy)$ be a division algebra with inversion, and $e\in A$ any non-zero element. Now the
isotope $A_{\rt_e\inv,\lt_e\inv}=(A,\ast)$ satisfies
$\lt_a^\ast=\lt_{\rt_e\inv(a)}\lt_e\inv$, so $\lt_A^\ast = \lt_A\lt_e\inv$.
Hence, by Lemma~\ref{hua},
$$ \left(\lt_a^\ast\right)\inv = \lt_e \lt_{\rt_e\inv(a)}\inv =
\lt_e \lt_{s(\rt_e\inv(a))} \lt_e \lt_e\inv \in \lt_A\lt_e\inv = \lt_A^\ast$$
so $A_{\rt_e\inv,\lt_e\inv}$ has inversion. Since $e^2$ is an identity element with
respect to the product $\ast$, by (1) we conclude
that $B = A_{\rt_e\inv,\lt_e\inv}$ is alternative, and $A=B_{\rt_e,\lt_e}$.

3.
Let $(B,xy)$ be a division algebra with inversion, $\a,\b\in\gl(B)$, and
$A=(B,\circ)= B_{\a,\b}$. For $a\in A$, we have $\lt_a^\circ=\lt_{\a(a)}\b$, so
$\lt_A^\circ=\lt_B\b$, and
$$ \left(\lt_a^\circ\right)\inv = \b\inv\lt_{\a(a)}\inv = \b\inv\lt_{s(\a(a))}\b\inv \b \,.$$
Hence $A$ has inversion if and only if $\b\inv \lt_B \b\inv = \lt_B$, equivalently,
$\b\lt_B\b = \lt_B$.

Suppose that $\b\lt_B\b = \lt_B$. Then there exists a linear map $\psi\in\gl(B)$ such that
$\b\lt_a\b = \lt_{\psi(a)}$ for all $a\in B$. Now $\b(a\b(b))=\psi(a)b$ for all $a,b\in B$.
Substituting $c=\b(b)$ gives $\b(ac) = \psi(a)\b\inv(c)$ for all $a,c\in B$.
Conversely it is clear that the last equation implies $\b\lt_{a}\b=\lt_{\psi(a)}$ for all
$a\in B$.
\end{proof}

\begin{cor} \label{real}
  Every division algebra with inversion that contains a non-zero idempotent is an isotope
  $B_{\a,\s}$, where $B$ is an alternative division algebra, $\a(1_B)=1_B$ and $\s$ is
  an automorphism of $B$ with $\s^2=\I_B$.
  Conversely, every isotope of this kind has inversion.
\end{cor}

\begin{proof}
Let $e\in A\setminus\{0\}$ be an idempotent.
By Proposition~\ref{general}(2), $(A,\circ)=B_{\a,\s}$, where $B$ is alternative with
identity element $e\circ e=e$, $\a=\rt_e^\circ$ and $\s=\lt_e^\circ$.
In particular, $\a(e)=\s(e)=e$.

Proposition~\ref{general}(3) now gives $\s(ab)=\psi(a)\s\inv(b)$ for all $a,b\in B$.
Inserting $b=e$ into this equation yields $\s(a)=\psi(a)$, that is, $\s=\psi$.
If instead $a=e$, then $\s(b)=\psi(1_B)\s\inv(b)=\s(e)\s\inv(b)=\s\inv(b)$, hence
$\s=\s\inv$.
This means that $\s(ab)=\s(a)\s(b)$ for all $a,b\in B$, so $\s\in\Aut(B)$.

The converse is clear, in view of Proposition~\ref{general}(3).
\end{proof}

\begin{rmk}
  Every  division algebra over $\R$ contains a non-zero idempotent
  \cite{segre}. Hence if such an algebra has inversion, it is isomorphic to an isotope of
  the type given in Corollary~\ref{real}.
  Any alternative real division algebra is isomorphic to either $\R$,
  $\C$, $\H$ or $\O$ \cite{zorn31}.
\end{rmk}

\begin{prop}  \label{1morph}
  Let $A_{\a,\s}=(A,\ast)$ and $B_{\b,\tau}=(B,\circ)$ be isotopes of alternative
  division algebras $(A,xy)$, $(B,x\cdot y)$ of the type given in Corollary~\ref{real}.
  Then
  $$ \{\f \in \Mor(A_{\a,\s},B_{\b,\tau}) \mid \f(1_A)=1_B \} =
  \{\f\in\Mor(A,B)\setminus\{0\} \mid \f\a=\b\f, \: \f\s=\tau\f \} $$
\end{prop}

\begin{proof}
It is readily verified that any morphism $\f:A\to B$ satisfying $\f\a=\b\f$ and
$\f\s=\tau\f$ is a morphism $A_{\a,\s} \to B_{\b,\tau}$.
If $\f$ is non-zero then $\f(1_A)=1_B$.

Suppose $\f\in\Mor(A_{\a,\s}, B_{\b,\tau})$ and $\f(1_A)=1_B$.
By definition,
\begin{equation}
\f(\a(x)\s(y))=\f(x\ast y)=\f(x)\circ\f(y)=\b\f(x)\cdot \tau\f(y)
\end{equation}
for all $x,y\in A$. Setting $y=1_A$ in this equation yields $\f\s=\tau\f$; setting $x=1_A$
yields $\f\a=\b\f$.
Consequently, $\f\a(x)\cdot\f\s(y) = \b\f(x)\cdot \tau\f(y) = \f(\a(x) \s(y))$ holds
for all $x,y\in A$. Since $\a$ and $\s$ are bijective, this means that $\f(x) \cdot \f(y)=\f(xy)$
for all $x,y\in A$, which proves that $\f$ is a morphism $A\to B$.
\end{proof}

We denote by $\mathscr{I}(k)$ the category of all division algebras over $k$
with inversion on the left.
Furthermore, $\mathscr{I}_n(k)$ denotes the full subcategory of $\mathscr{I}(k)$ formed by
all objects of dimension $n$,  $\mathscr{I}^1(k)$ the full subcategory formed by all
object with left unity, and $\mathscr{I}^1_n(k)=\mathscr{I}^1(k)\cap\mathscr{I}_n(k)$.
As it will turn out to be convenient in the sequel, we stipulate that morphisms in $\Ii$
are non-zero.

\subsection{The real, four-dimensional case}
\label{4dim}

In this section, we investigate real division algebras of dimension four that have
inversion on the left.
Over $\R$, the quaternion algebra $\H$ is the only alternative division algebra of
dimension four, so by Proposition~\ref{general}, the algebras of interest are of the form
$A=\H_{\a,\b}$ with $\a,\b\in\gl(\H)$ such that $\b\lt_{\H}\b=\lt_\H$.

In any associative algebra or ring $A$, the set of invertible elements is denoted by
$A^\times$.
We denote by $\hat{\H}$ the group of quaternions of norm one, and set
$\go(V,q)=\{\f\in\gl(V) \mid \exists_{\mu\in k^\times}:\: q\f=\mu q\}$.
The group $\{\I,\kappa\}\subseteq\og(\H)$ acts on $\gl(\H)$ by $\gamma^\l=\l\gamma\l$
($\gamma\in\gl(\H)$, $\l\in\{\I,\kappa\}$), and $\lt_a^\kappa=\rt_{a\inv}$,
$\rt_a^\kappa=\lt_{a\inv}$ for $a\in\hat{\H}$.

\begin{prop} \label{criterion}
  Let $\a,\b\in\gl(\H)$. The isotope $\H_{\a,\b}$ has inversion if and only if
  $\b=\lt_a\rt_u$ for some $a,u\in\H^\times$, with $u^2\in\R1_\H$.
\end{prop}

\begin{proof}
The ``if'' part of the proposition is easily verified.
Suppose instead that $\b\lt_\H\b=\lt_\H$.
Set $\tilde{\b}=\b$ if $\det\b>0$, and $\tilde{\b}=\b\kappa$ if $\det\b<0$.
Now, polar decomposition (see \eg{} \cite[\S14]{gantmacher59i}), gives rise to a unique
decomposition of $\tilde{\b}$ as $\tilde{\b}=\gamma\d$, where $\gamma\in\so(\H)$ and
$\d\in\Pds(\H)$. The map $\gamma$ can in turn be written as $\lt_a\rt_u$ for some
$a,u\in\hat{\H}$ \cite[Corollary~9]{hurwalg}.
Hence we have $\b=\lt_a\rt_u\d\l$, with $\l\in\{\I,\kappa\}$, and this decomposition is
unique up to simultaneous change of sign of the elements $a,u\in\H$.

For $x\in\H$, we have
\begin{equation*}
\b\lt_x\b = \lt_a\rt_u\d\l \lt_x \lt_a\rt_u\d\l = \lt_a\rt_u \lt_x^\l\lt_a^\l\rt_u^\l\e\d^\l
\end{equation*}
where
$\e=\left(\lt_x^\l\lt_a^\l\rt_u^\l\right)\inv\!\d\lt_x^\l\lt_a^\l\rt_u^\l \in\Pds(\H)$.
Since $\lt_y,\rt_y,\lt_y^\l,\rt_y^\l\in\go(\H)$ for all $y\in\H$ and
$\b\lt_x\b\in\lt_\H\subseteq \go(\H)$, we get
$\e\d^\l\in\go(\H)$. The uniqueness of polar decomposition now implies
$\e\d^\l=\mu\I$ for some $\mu>0$.
Thus $\b\lt_x\b = \mu\lt_a\rt_u \lt_x^\l\lt_a^\l\rt_u^\l$.

If $\l=\kappa$, then $\b\lt_x\b=\mu\lt_{au\inv}\rt_{a\inv x\inv u}$, which is in $\lt_\H$
if and only $a\inv x\inv u\in\R1_\H$. Conseqently, $\b\lt_\H\b\not\subseteq\lt_\H$ in this
case.
Hence $\l=\I$, and we have $\b\lt_x\b = \lt_{axa}\rt_{u^2}$, which is in $\lt_A$ if
and only if $u^2\in\R1_\H$.

Next, the identity
$\left(\lt_x\lt_a\rt_u\right)\inv\!\d\lt_x\lt_a\rt_u\d^\l=\e\d^\l=\mu\I$ implies
that $\lt_x\inv\d\lt_x=\mu\lt_a\rt_u\d\inv\rt_u\inv\lt_a\inv$. Since the right hand side
is independent of $x$, this is only possible if $\d$ commutes with every element in
$\lt_\H$, that is, $\d=\rho\I$ for some scalar $\rho>0$.
Hence, $\b=\rho\lt_a\rt_u=\lt_{\rho a}\rt_u$, as asserted.
\end{proof}

Set $U_\H=\{u\in\H^\times\mid u^2\in\R1\}/\R^\times$.
The group $\H^\times/\R^\times$ acts on the set
$\mathscr{C} = \H^\times/\R^\times\times\spds(\H)\times U_\H$
by
\begin{equation}\label{Caction}
s\cdot(a,\d,u)=(c_s(a),c_s\d c_s\inv,c_s(u))
\end{equation}
where $c_s=\lt_s\rt_s\inv$.
The results in \cite[Section~3]{hurwalg} now give the following corollary.

\begin{cor}
  The category $\mathscr{I}_4(\R)$ decomposes as a coproduct
$$\mathscr{I}_4(\R)=\mathscr{I}_4(\R)_1\amalg\mathscr{I}_4(\R)_{-1}$$
where, for $i=\pm1$, $\mathscr{I}_4(\R)_i\subseteq\mathscr{I}_4(\R)$ is the full subcategory
formed by all objects $A$ satisfying $\sign(\det(\rt_x))=i$ for all $x\in A\setminus\{0\}$.
For each $i=\pm1$, the functor
$\mathscr{M}_i:{}_{\H^\times\!/\R^\times}\mathscr{C}\to\mathscr{I}_4(\R)_i$
defined for morphisms by $\mathscr{M}_i(s)=c_s$, and
$$\mathscr{M}_1(a,\d,u)=\H_{\lt_a\!\d,\rt_u} \,,\qquad
\mathscr{M}_{-1}(a,\d,u)=\H_{\rt_a\!\d\kappa,\rt_u}$$
for objects, is an equivalence.
\end{cor}

\begin{proof}
This follows from Theorem~6 and Propositions~10, 11 and 12 in \cite{hurwalg}.

Observe that $\H_{\a,\b}\in\mathscr{I}_4(\R)$ belongs to $\mathscr{I}_4(\R)_i$ if and only
if $\sign(\det(\a))=i$, $i=\pm1$. Thus this decomposition coincides the one given by
\cite[Proposition~10]{hurwalg}.
The groupoid ${}_{\H^\times/\R^\times}\mathscr{C}$ may be viewed as a full subcategory of
$\mathscr{Z}$ defined in \cite[Proposition~12]{hurwalg} via the embedding
$(a,\d,u)\mapsto ((a,u),(\d,\I))$.
Now our Proposition~\ref{criterion}, together with Propostion~11, 12 and Theorem~6 of
\cite{hurwalg} gives the result.
\end{proof}

Our aim is to give a set of representatives for the $\H^\times\slash\R^\times$-orbits of
$\mathscr{C}$, thereby completing the classification of
the four-dimensional real division algebras with inversion.
Remember that $s\mapsto c_s$ defines an epimorphism $\H^\times\to\Aut(\H)$ with
kernel $\R^\times$, inducing an isomorphism $\H^\times/\R^\times\to\Aut(\H)$.
Below, we shall view $\mathscr{C}$ as an $\Aut(\H)$-set with action defined through
this isomorphism.
Observe also that $\so(\Im\H)\to\Aut(\H),\; f\mapsto \I_\R\oplus f$ is an isomorphism.
Choosing an orthonormal basis, we may identify $\Im\H$ with the Euclidean space
$\R^3$, and hence $\H$ with $\R\oplus\R^3$.
For $\d\in\spds(\H)$, we have
$\d=\begin{pmatrix} \b & b^\ast \\ b & B \end{pmatrix}$ for some $\b\in\R$,
$b\in\R^3$ and $B\in\R^{3\times3}$.
Now conjugation with $\f=\I_\R\oplus f\in \Aut(\H)$ is given by
$$\f\d\f\inv = \begin{pmatrix}  \b & f(b)^\ast \\ f(b) & fBf\inv \end{pmatrix}\,.$$
The matrix $B$ is symmetric, and hence is diagonalisable under conjugation with
$\so(\R^3)$. 

Set $\hat{\mathscr{C}}= \H^\times/\R^\times \times \sym(\H) \times U_\H$, with an
$\Aut(\H)$-action given by $\f\cdot(a,\d,u)=(\f(a),\f\d\f\inv,\f(u))$.
Clearly, $\mathscr{C}\subseteq\hat{\mathscr{C}}$ as an $\Aut(\H)$-set.
The point of introducing this larger set is that we can find normal forms for
$\hat{\mathscr{C}}$ and restrict back to $\mathscr{C}$, which turns out to have certain
technical advantages over considering $\mathscr{C}$ directly.

Define $\so(\R^3)$-sets $\mathscr{B}_{ij}$, $i,j\in\{0,1\}$ by
\begin{align*}
  \mathscr{B}_{00}&= \mathbb{P}(\R^3)\times \R^3\times\sym(\R^3)\times\R\,, &
  \mathscr{B}_{10}&=
  \mathbb{P}(\R^3)\times \mathbb{P}(\R^3)\times \R^3\times \sym(\R^3)\times\R\,, \\
  \mathscr{B}_{01}&= \R^3\times \R^3\times\sym(\R^3)\times\R\,, &
  \mathscr{B}_{11}&= \mathbb{P}(\R^3)\times \R^3\times \R^3\times \sym(\R^3)\times\R\,. \\
\end{align*}
and $\so(\R^3)$-actions
\begin{align}
  f\cdot(c,b,B,\b) &= (f(c),f(b),fBf\inv,\b) &
  &\mbox{on $\mathscr{B}_{00}$ and $\mathscr{B}_{01}$}, \label{so3act1} \\
  f\cdot(u,c,b,B,\b) &= (f(u),f(c),f(b),fBf\inv,\b) &
  &\mbox{on $\mathscr{B}_{10}$ and $\mathscr{B}_{11}$}. \label{so3act2}
\end{align}
Now define functors
$\mathscr{N}_{ij}:{}_{\so(\R^3)}\mathscr{B}_{ij}\to{}_{\Aut(\H)}\hat{\mathscr{C}}$ by
\begin{align*}
  \mathscr{N}_{00}(c,b,B,\b)&=\left(\smatr{0\\c},\smatr{\b&b^\ast\\b&B},1 \right), &
  \mathscr{N}_{10}(u,c,b,B,\b)&=\left(\smatr{0\\c},\smatr{\b&b^\ast\\b&B},u \right), \\
  \mathscr{N}_{01}(c,b,B,\b)&=\left(\smatr{1\\c},\smatr{\b&b^\ast\\b&B},1 \right), &
  \mathscr{N}_{11}(u,c,b,B,\b)&=\left(\smatr{1\\c},\smatr{\b&b^\ast\\b&B},u \right).
\end{align*}
and $\mathscr{N}_{ij}(f)=\I_\R\oplus f$ for $f\in\so(\R^3)$.
It is easy to see that each $\mathscr{N}_{ij}$ is full and faithful, and that
$\hat{\mathscr{C}}= \coprod_{i,j\in\{0,1\}}\mathscr{N}_{ij}(\mathscr{B}_{ij})$.

We proceed to find normal forms for the $\so(\R^3)$-sets $\mathscr{B}_{ij}$.
Let $\hat{\mathcal{T}}=\{d\in\R^3 \mid d_1\le d_2 \le d_3\}$,
$\mathcal{T}=\{d\in\hat{\mathcal{T}} \mid 0<d_1 \}$ and
$D_d=\smatr{d_1\\&d_2\\&&d_3}$ for any $d\in\R^3$.
From the real spectral theorem follows that every orbit in $\mathscr{B}_{ij}$,
$i,j\in\{0,1\}$ contains an element for which the matrix $B\in\sym(\R^3)$ takes the form
$D_d$ for some $d\in\hat{\mathcal{T}}$ ($d_1,d_2,d_3$ being the eigenvalues of $B$), and
$d$ is an invariant for the orbit.
The $\so(\R^3)$-action on $\sym(\R^3)$ induced by either of (\ref{so3act1}) and
(\ref{so3act2}) is given by conjugation, so its isotropy subgroup
$\so_d(\R^3)\subseteq\so(\R^3)$ at $D_d$ consists of all $f\in\so(\R^3)$ that leave the
eigenspaces of $D_d$ invariant.
Thus
\begin{equation*}
  \so_d(\R^3)=
  \begin{cases}
    \so(\R^3) & \mbox{for all} \quad
    d\in\hat{\mathcal{T}}_1=\{d\in\hat{\mathcal{T}} \mid d_1=d_2=d_3\}, \\
    \langle\iota(\so(\R^2)),\bar{\iota}(-\I_2)\rangle & \mbox{for all} \quad
    d\in\hat{\mathcal{T}}_2=\{d\in\hat{\mathcal{T}} \mid d_1=d_2<d_3\}, \\
    \langle\bar{\iota}(\so(\R^2)),\iota(-\I_2)\rangle & \mbox{for all} \quad
    d\in\hat{\mathcal{T}}_3=\{d\in\hat{\mathcal{T}} \mid d_1<d_2=d_3\}, \\
    \langle\iota(-\I_2),\bar{\iota}(-\I_2)\rangle & \mbox{for all} \quad
    d\in\hat{\mathcal{T}}_4=\{d\in\hat{\mathcal{T}} \mid d_1<d_2<d_3\},
  \end{cases}
\end{equation*}
where $\iota$ and $\bar{\iota}$ are the embeddings $\so(\R^2)\to\so(\R^3)$ given by
$$ \iota(\f)= \left(
\begin{array}{c|c}
  \f&\mbox{$\begin{array}{c}0\\0\end{array}$} \\ \hline
  \rule{0pt}{10pt}\mbox{$\begin{array}{cc}0&0\end{array}$}&1\rule{0pt}{10pt}
\end{array}
\right) \quad \mbox{and}\quad
\bar{\iota}(\f)= \left(
\begin{array}{c|c}
  1&\mbox{$\begin{array}{cc}0&0\end{array}$}    \\ \hline
  \mbox{$\begin{array}{c}\rule{0pt}{10pt}0\\0\end{array}$} & \f
\end{array}
\right)$$
respectively.

To obtain normal forms for the $\so(\R^3)$-sets $\mathscr{B}_{ij}$, it suffices to
consider, for each $d\in\hat{\mathcal{T}}$, the action of $\so_d(\R^3)$ on elements of the
form $(c,b, D_d,\b)$ in $\mathscr{B}_{0i}$ and $(u,c,b,D_d,\b)$ in $\mathscr{B}_{1i}$.
As $D_d$ and $\b$ are fixed by $\so_d(\R^3)$, the essential problem is to find normal
forms for the pairs $(c,b)$ and $(u,c,b)$ under the actions induced from (\ref{so3act1})
and (\ref{so3act2}).
This is an elementary, though rather technial, task.
Normal forms  are listed in Appendix \ref{normalforms}. On this basis, we may state the
following result.
For $s\in\{1,2,3,4\}$, set
$\hat{\mathcal{D}}_s=\{D_d\mid d\in\hat{\mathcal{T}}_s\}\subseteq\sym(\R^3)$ and
$\mathcal{D}_s=\{D_d\mid d\in\mathcal{T}_s\}\subseteq\Pds(\R^3)$, and let
$\mathcal{N}_{ij}^s$ ($i,j\in\{0,1\}$, $s\in\{1,2,3,4\}$) be defined as in Appendix~\ref{normalforms}.

\begin{prop} \label{hatclassification}
  For all $i,j\in\{0,1\}$, the set
  $$\bigcup_{s=1}^4\left(\mathcal{N}_{ij}^s\times \hat{\mathcal{D}}_s\times\R\right) \,,$$
  is a cross-section for the orbit set of the $\so(\R^3)$-set
  $\mathscr{B}_{ij}$.
\end{prop}

By Proposition~\ref{hatclassification}, the category $\hat{\mathscr{C}}$ is classified up to
isomorphism. To classify $\mathscr{C}$, and thereby $\mathscr{I}_4(\R)$, it now suffices
to determine which elements in the classifying list of $\hat{\mathscr{C}}$ belong to
$\mathscr{C}$.

\begin{cor} \label{classification}
  For all $j\in\{0,1\}$, the sets
  \begin{align*}
    \mathcal{S}_{0j}&=\left\{(c,b,D_d,\b)\in\bigcup_{s=1}^4\left(\mathcal{N}_{0j}^s\times
        \mathcal{D}_s\times\R_{>0}\right) \mid
      \det \begin{pmatrix} \b & b^\ast \\ b & D_d \end{pmatrix} =1 \right\} \\
    \intertext{and}
    \mathcal{S}_{1j}&=\left\{(u,c,b,D_d,\b)\in\bigcup_{s=1}^4\left(\mathcal{N}_{1j}^s\times
        \mathcal{D}_s\times\R_{>0}\right) \mid
      \det \begin{pmatrix} \b & b^\ast \\ b & D_d \end{pmatrix} =1 \right\}
  \end{align*}
  are cross-sections for the orbit sets of ${}_{\so(\R^3)}\mathscr{B}_{0j}$ and
  ${}_{\so(\R^3)}\mathscr{B}_{1j}$ respectively. Hence
  $$\bigcup_{\genfrac{}{}{0pt}{} {r=\pm1}{i,j\in\{0,1\}} }
      \mathscr{M}_r\left(\mathscr{N}_{ij}(\mathcal{S}_{ij})\right)$$
  is a cross-section for the isomorphism classes in $\mathscr{I}_4(\R)$.
\end{cor}

We remark that $\det \begin{pmatrix} \b & b^\ast \\ b & D_d \end{pmatrix} =
\b d_1d_2d_3 -b_1^2d_2d_3 -d_1b_2^2d_3 -d_1d_2b_3^2$.

\begin{proof}
An element $(a,\d,u)\in\hat{\mathscr{C}}$ is in $\mathscr{C}$ if and only if $\d$ is
positive definite and $\det(\d)=1$. By Sylvester's criterion,
$\d=\begin{pmatrix} \b & b^\ast \\ b & D_d \end{pmatrix}$ is positive definite if and only
if all its principal minors are positive.
It is easy to see that this is equivalent to the determinant and the diagonal elements of
$\d$ being positive. Hence $d\in\mathcal{T}$, so $D_d\in\mathcal{D}_s$ for some $s$,
and $\b>0$. Also, by assumption, $\det\d=1$.
\end{proof}

It is not difficult to see that
$\mathscr{M}_r\left(\mathscr{N}_{ij}(\mathcal{S}_{ij})\right)$ is left unital if and only
if $i=0$. Hence it follows from Corollary~\ref{classification} that
$$\bigcup_{\genfrac{}{}{0pt}{}{r=\pm1}{j\in\{0,1\}}} \mathscr{M}_r\left(\mathscr{N}_{0j}(\mathcal{S}_{0j})\right)$$
classifies $\mathscr{I}_4^1(\R)$.
In the next section, an alternative approach to the left unital case is taken,
aiming for an understanding of the eight-dimensional case.

\subsection{The left unital case} \label{leftunital}

\begin{lma} \label{lunitylemma}
  Let $(A,xy)$ be a unital $k$-algebra, $\a\in\gl(A)$, and $\s\in\Aut(A)$ an automorphism such
  that $\s^2=\I$.
  Then $A_{\a,\s}$ has left unity if and only if $\s=\I$.
\end{lma}
\begin{proof}
Write $(A,x\circ y)=A_{\a,\s}$. For any $x\in A$, we have $\lt_x^\circ=\lt_{\a(x)}\s$, thus
$\lt^\circ_x=\I$ if and only if $\lt_{\a(x)}=\s\inv=\s$. Then $\lt_{\a(x)}\in\Aut(A)$,
implying $\a(x)=\lt_{\a(x)}(1_A)=1_A$, hence $\s=\lt_{1_A}=\I$.
Conversely, if $\s=\I$ then $\a\inv(1_A)$ is a left unity in $A_{\a,\s}$.
\end{proof}

By Corollary~\ref{real} and Lemma~\ref{lunitylemma}, every object in $\Ii^1(\R)$ is
isomorphic to some $A_{\a,\I}$, where $A$ is alternative and $\a(1_A)=1_A$. Since the left
unity is preserved by any algebra morphism, Proposition~\ref{1morph} gives a full
characterisation of the morphisms between objects in $\Ii^1(\R)$.

A \emph{vector product algebra} is a Euclidean space $V=(V,\langle\,\rangle)$ endowed with an
anti-commutative algebra structure $\pi$ satisfying $\langle\pi(u,v),v\rangle=0$ and
$\langle\pi(u,v),\pi(u,v)\rangle=\langle u,u\rangle\langle v,v\rangle-\langle u,v\rangle^2$.
The category of vector product algebras $(V,\pi)$ (morphisms in which are orthogonal
algebra morphisms)
 is equivalent to the category of real
alternative division algebras; an equivalence is given by the construction $(V,\pi)\mapsto
A{(V,\pi)}$, where $A{(V,\pi)}=\R\times V$ with multiplication
$$(\l,v)(\mu,w) = (\l\mu - \langle v,w\rangle \,,\, \l w + \mu v + \pi(v,w))$$
for objects, and $\f\mapsto\I_\R\times\f$ for morphisms.

For any Euclidean space $V$, denote by $V^\ast$ its dual space.
Let $\mathscr{V}$ be the class of objects $(V,\b,\d,\s)$, where $V=(V,\pi_V)$ is a vector
product algebra, and $(\b,\d,\s)\in\og(V)\times\Pds(V)\times V^\ast$.
A morphism $\f:(V,\b,\d,\s)\to(W,\gamma,\e,\tau)$ between objects in $\mathscr{V}$ is
defined to be a morphism $\f:(V,\pi_V)\to(W,\pi_W)$ satisfying $\f\b=\gamma\f$,
$\f\d=\e\f$ and $\s=\tau\f$. This gives $\mathscr{V}$ the structure of a category.

\begin{prop}
  The categories $\mathscr{V}$ and $\Ii^1(\R)$ are equivalent. An equivalence $\mathscr{F}$
  is given by $\mathscr{F}(V,\b,\d,\s)=A(V,\pi_V)_{\a,\I}$, where the linear map
  $\a\in\gl(A(V,\pi_V))$ is given as
  $$\a=\begin{pmatrix}
    1 & \s \\
    0 & \d\b
  \end{pmatrix}
  :\R\times V \to \R\times V \,,$$
  and $\mathscr{F}(f)=\I_\R\times f$ for morphisms.
\end{prop}

\begin{proof}
Clearly, $\mathscr{F}(X)\in\Ii^1(\R)$ for all $X\in\mathscr{V}$.
Let $X=(V,\b,\d,\s)$ and $Y=(W,\gamma,\e,\tau)$ be elements in $\mathscr{V}$, and let
$f:X\to Y$ be a morphism in $\mathscr{V}$. We have $\mathscr{F}(X)=A(V,\pi_V)_{\psi,\I}$ and
$\mathscr{F}(Y)=A(W,\pi_W)_{\chi,\I}$ with
$$\psi = \begin{pmatrix} 1&\s\\ 0&\d\b \end{pmatrix} \in\gl(\R\times V)\:,\quad
\chi = \begin{pmatrix} 1&\tau\\ 0&\e\gamma \end{pmatrix}
\in\gl(\R\times W)\,.$$
Clearly, $\mathscr{F}(f)(1_{\mathscr{F}(X)}) = 1_{\mathscr{F}(Y)}$, and it is straightforward
to verify that $\mathscr{F}(f)\psi=\chi\mathscr{F}(f)$. Since $\mathscr{F}(f)$ is a
morphism $A(V,\pi_V)\to A(W,\pi_W)$ of alternative division algebras,
Proposition~\ref{1morph} now implies that
$\mathscr{F}(f):\mathscr{F}(X)\to\mathscr{F}(Y)$ is a morphism in $\Ii^1(\R)$.
This shows that $\mathscr{F}$ is functorial.

Every $A\in\Ii^1(\R)$ is isomorphic to an isotope $B_{\a,\I}$ of an alternative
division algebra $B$, with $\a(1_B)=1_B$. Let $V=\Im B$, and $\pi_V$ correspondingly
defined by $\pi_V(v,w)=\frac{1}{2}[v,w]$. This construction is quasi-inverse to
$(V,\pi)\mapsto A(V,\pi)$, thus $A(V,\pi_V)\simeq B$ (see \eg{} \cite{bg67}).
Now $\a$ gives rise to a linear form $\s:V\to\R$ and a linear endomorphism
$\tilde{\a}:V\to V$ such that $\a(v)=\s(v)1_B+\tilde{\a}(v)$ for all $v\in V=\Im B$. Since
$\det(\a)=\det(\tilde{\a})$, the map $\tilde{\a}$ is invertible.
Polar decomposition yields a unique pair $(\b,\d)\in\og(V)\times\Pds(V)$ such that
$\tilde{\a}=\d\b$.
It is now easy to verify that $\mathscr{F}(V,\b,\d,\s)\simeq B_{\a,\I}\simeq A$. Thus the
functor $\mathscr{F}:\mathscr{V}\to\Ii^1(\R)$ is dense.

Assume $\f:\mathscr{F}(X)\to\mathscr{F}(Y)$ is a morphism in $\Ii^1(\R)$. The left unity of
$\mathscr{F}(X)=A(V,\pi_V)_{\psi,\I}$ is the element $1_{A(V,\pi_V)}$, and it is mapped by
$\f$ to the left unity $1_{A(W,\pi_w)}$ in $\mathscr{F}(Y)=A(W,\pi_W)_{\chi,\I}$. Thus
Proposition~\ref{1morph} applies, giving that $\f$ is a morphism $A(V,\pi_V)\to A(W,\pi_W)$
and $\f\psi=\chi \f$. This means that the restricted and corestricted map
$f=\f|_V^W:V\to W$ is a morphism $(V,\pi_V)\to(W,\pi_W)$ of vector product
algebras.
Moreover, $\f\psi=\chi \f$ implies $\s=\tau f$ and $f\d\b=\e\gamma f$.
In particular, $\im(\e\gamma f)\subseteq\im(f)$. Since non-zero morphisms of division
algebras are always injective, we may apply $f\inv$ from the left to obtain
$\d\b=f\inv\e\gamma f=(f\inv\e f)(f\inv\gamma f)$. As $f\inv\e f\in\Pds(V)$ and
$f\inv\gamma f\in\og(V)$, uniqueness of the polar decomposition implies
$\d=f\inv\e f$ and $\b=f\inv\gamma f$, that is, $f\d=\e f$ and $f\b=\gamma f$.
Hence $f\in\Mor_{\mathscr{V}}(X,Y)$, and $\f=\mathscr{F}(f)$, so $\mathscr{F}$ is full.

Faithfulness is immediate from the definition of $\mathscr{F}$.
\end{proof}

\begin{cor} \label{lunital}
  Let $n\in\{0,1,3,7\}$.
  The category $\Ii^1_{n+1}(\R)$ is equivalent to the groupoid $\mathscr{X}_n$ of the action
  of $\Aut(\R^n,\pi_n)$ on $\og(\R^n)\times\Pds(\R^n)\times \R^n$ given by
  \begin{equation} \label{gruppverkan}
    f\cdot(\b,\d,s)= (f\b f\inv,f\d f\inv,f(s)) \,.
  \end{equation}
\end{cor}

\begin{proof}
Since morphisms in $\Ii(\R)$ are injective, all morphisms in $\Ii^1_{n+1}(\R)$ are
isomorphisms, and so are all morphisms in the full subcategory
$\mathscr{V}_n=\mathscr{F}\inv(\Ii_{n+1}^1(\R))\subseteq\mathscr{V}$.
Thus a morphism $(\R^n,\b,\d,\s)\to(\R^n,\gamma,\e,\tau)$ in $\mathscr{V}$ is an element
$f\in\Aut(\R^n,\pi_n)$ satisfying $f\b=\gamma f$, $f\d=\e f$ and $\s=f\tau$.
An equivalence $\mathscr{X}_n\to\mathscr{V}_n$ is given by
$(\b,\d,s)\mapsto(\R^n,\b,\d,\s_s)$, where $\s_s(x)=\langle s,x\rangle$ for objects and $f\mapsto f$
for morphims $f\in\Aut(\R^n,\pi_n)$.
\end{proof}

The groupoid $\mathscr{X}_7$ is a $56$-dimensional manifold, acted upon
by the $14$-dimensional Lie group $\mathscr{G}_2=\Aut(\R^7,\pi_7)$.
Finding a cross-section for the orbit set of this group action would, although
theoretically possible, be a very difficult task.
Normal forms for elements in $\mathscr{X}_7$ of the form $(\b,\I,0)$ or $(\I,\d,0)$
can be obtained from the normal forms given in \cite{ava} and
\cite{nform} for the $\mathscr{G}_2$-action by conjugation on the sets $\og(\R^7)$
and $\Pds(\R^7)$ respectively.

\section{Division algebras with involutive inversion}  \label{sec:invinv}

A division algebra $A$ is said to have \emph{involutive inversion} if it has inversion on the left and the inversion map $s\colon A\setminus \{0\} \to A \setminus\{0 \}$ satisfies
\begin{displaymath}
 s(ab) = s(b)s(a)
\end{displaymath}
for all $a,b \in A\setminus\{0\}$.  In this case $y = s(x(s(x)s(y))) = (yx)s(x)$, so
\begin{displaymath}
 \rt^{-1}_{a} = \rt_{s(a)}
\end{displaymath}
for any $a \neq 0$. In particular, the opposite algebra $A\op$ of a division algebra $A$
with involutive inversion is also a division algebra with involutive inversion.

\begin{prop}
\label{prop:algebras_with_inversion}
 An algebra $(A,xy)$ is a division algebra with involutive inversion and a non-zero idempotent
 if and only if $xy = \tau(x) * \sigma(y)$ for some alternative division algebra $(A,*)$
 and automorphisms $\sigma,\tau$ of $(A,*)$ with $\sigma^2 = \tau^2 = \I_A$ and
 $\sigma \tau \sigma = \tau \sigma \tau$.
\end{prop}
\begin{proof}
 By Corollary \ref{real} we can write the product on $A$ as $xy = \alpha(x)* \tau(y)$ for
 some alternative division algebra $(A,*)$ and an automorphism $\tau = \lt_e$ of
 $(A,*)$ with $\tau^2 = \I_A$. Working with $A\op$ we obtain that $xy = \sigma(x)*\tau(y)$
 where $\sigma = \rt_e$ is another automorphism of $(A,*)$ satisfying
$\sigma^2=\I_A$. The condition $s(ab) = s(b)s(a)$ for all $a,b \ne 0$ is equivalent to
$\sigma \tau \sigma = \tau \sigma \tau$.
\end{proof}

Below, we shall often encounter isotopes $A_{\s,\tau}$ for which $\s^2=\I_A$ and
$\s\tau\s=\tau\s\tau$. We record the following observation.

\begin{lma}
\label{lem:group}
Let $G$ be a non-trivial group with unit element $e$, generated by $x$ and $y$ with  $x^2 = e$ and $xyx= yxy$.  Then {$y^2 =e$ and} $G$ is a quotient of $\mathrm{D}_6$, the dihedral group with 6 elements, and hence
isomorphic either to $\mathrm{D}_6$ itself or to the cyclic group $\cc_2$.
\end{lma}

Division algebras with involutive inversion and a non-zero idempotent are quite close to
alternative algebras. The following proposition characterizes these algebras
in terms of a quasigroup identity. 

\begin{prop} \label{prop:identity}
  Let $(A,xy)$ be a division algebra with a non-zero idempotent over a field of characteristic
  different from $2$. Then $A$ has involutive inversion if and only if it satisfies the
  identity
  \begin{equation}
    \label{eq:identity}
    x((yz)(xt)) = ((xy)(zx))t \,.
  \end{equation}
\end{prop}
\begin{proof}
From the description of division algebras with involutive inversion and non-zero
idempotents in Proposition \ref{prop:algebras_with_inversion}, it is easily verified that
these algebras satisfy the identity  $x((yz)(xt)) = ((xy)(zx))t$.

Conversely, let $e$ be a non-zero idempotent of a division algebra $A$ that satisfies the identity   $x((yz)(xt)) = ((xy)(zx))t$. Evaluating this identity at $x = y = z = e$ we obtain that $e(et) = t$ so $\lt^2_{e} =\I_A$. The evaluation of the same identity at $x = y = t = e$ gives $\lt_e \rt_e \lt_e = \rt_e \lt_e \rt_e$. By Lemma \ref{lem:group} this implies that $\rt^2_e = \I_A$.

Consider the new product $x*y = (xe)(ey)$ on $A$. The element $e$ is the unit element of
the new algebra $(A,*)$ and the new left multiplication operator by $x$ is
$\lt^*_x = \lt_{xe}\lt_e$. The identity $x((yz)(xt)) = ((xy)(zx))t$ implies that
$\lt_x\lt_{yz}\lt_x = \lt_{(xy)(zx)}$, so
$(\lt^*_x)^2= (\lt_{xe}\lt_e\lt_{xe})\lt_e= \lt_{x(e(xe))}\lt_e$ is again a left
multiplication operator on $(A,x*y)$. Hence, the existence of unit element in $(A,x*y)$
implies $(\lt^*_x)^2 = \lt_{x*x}$, and since $(A,x*y)$ is a division algebra we can conclude
that it is alternative.

Observe that $e(x*y) = e((xe)(ey)) = ((ex)e)y = (ex) *(ey)$, thus $\lt_e$ is an automorphism of $(A,x*y)$. The identity $e((xe)(ey)) = ((ex)e)y$ implies that $(\lt_e,\rt_e\lt_e\rt_e,\lt_e) \in \Atp(A)$ {(see (\ref{eq:Atp}) below)}. Let $\alpha = \rt_e\lt_e\rt_e = \lt_e \rt_e \lt_e$. Evaluating (\ref{eq:identity}) at $x = t = e$ we get that $e((yz)e) =((ey)(ze))e$ so $\alpha(yz) = (ey)(ze)$. Therefore $(\alpha, \lt_e,\rt_e)\in\Atp(A)$ and
\begin{align*}
 (\rt_e,\rt_e,\alpha) &= (\lt_e\alpha \lt_e, \alpha \lt_e \alpha, \lt_e \rt_e \lt_e) \\
 &=  (\lt_e,\alpha,\lt_e)(\alpha,\lt_e,\rt_e)(\lt_e,\alpha,\lt_e)\in\Atp(A) \,.
\end{align*}
Hence $(x*y)e = ((xe)(ey))e = ((xe)e)\alpha(ey) = x(e(ye)) = (xe) * (ye)$.
\end{proof}

\begin{cor}
 A real division algebra $(A,xy)$ has involutive inversion if and only if it satisfies the identity
\[
  x((yz)(xt)) = ((xy)(zx))t \,.
\]
\end{cor}

We now proceed to classify all real division algebras with involutive inversion. 
Proposition~\ref{invinv} gives the classification result, and its proof occupies the
remainder of this section.
In stark contrast with the more general situation treated in Section~\ref{sec:inversion},
the real division algebras with involutive inversion comprise only finitely many
isomorphism classes in each dimension. 

A \emph{Cayley triple} in $\O$ is a triple $(u,v,z)$ of unit vectors in $\Im(\O)$
such that $u,v,uv,z$ are mutually orthogonal. The automorphism group of $\O$ acts simply
transitively on the set of Cayley triples in $\O$, by
$\f\cdot(u,v,z)=(\f(u),\f(v),\f(z))$, \ie, fixing a Cayley triple $(i,j,l)$, there is a
bijection between $\Aut(\O)$ and the set of Cayley triples in $\O$, given by 
$\f\mapsto\f\cdot(i,j,l)$.
Similarly for the quaternions, $\Aut(\H)$ acts simply transitively on the set of
orthonormal pairs in $\Im(\H)$, so every automorphism of $\H$ is uniquely determined by
the images of the standard basis vectors $i$ and $j$. 

For ease of notation, we set 
$(c,s)=\left(\cos(2\pi/3)\,,\,\sin(2\pi/3)\right)=\left(-1/2\,,\,\sqrt{3}/2\right)$.

\begin{prop} \label{invinv}
  There exist precisely ten isomorphism classes of real division algebras with involutive
  inversion, each isomorphic to an isotope $B_{\s,\tau}$ of an alternative division algebra
  $B$, with $\s,\tau\in\Aut(B)$ as follows:
  \begin{align}
    \intertext{If $B=\R$:}
    &\s=\tau=\I_\R. \label{1} \\
    \intertext{If $B=\C$:} 
    &\s=\tau=\I_\C \quad \mbox{or} \\
    &\s=\tau=\kappa.\\
    \intertext{If $B=\H$:}
    &\s=\tau=\I_\H, \label{4} \\
    &\s\cdot(i,j)=\tau\cdot(i,j)=(i,-j), \quad\mbox{or} \label{22} \\
    &\s\cdot(i,j)=(i,-j),\:\tau\cdot(i,j)=(ci+sj,si-cj). \label{111}\\
    \intertext{If $B=\O$:}
      &\s=\tau=\I_\O, 
      \label{8} \\
      &\s\cdot(i,j,l)=\tau\cdot(i,j,l)=(i,j,-l), 
      \label{44} \\
      &\begin{cases}
        \s\cdot(i,j,l) = (-i,-j,l),\\
        \tau\cdot(i,j,l) = (-i,-j,cl+sil), \quad\mbox{or} \\
      \end{cases}
      \label{222}\\
      &
      \begin{cases}
        \s\cdot(i,j,l) = (i,j,-l), \\
        \tau\cdot(i,j,l) = (ci+sil,cj+sjl,-l).
      \end{cases}
      \label{113}
  \end{align}
\end{prop}

It is easy to check that all the pairs of automorphisms $\s$ and $\tau$ defined in
Proposition~\ref{invinv} satisfy the conditions 
\begin{equation} \label{d6}
\s^2=\tau^2=(\tau\s)^3=\I_B,
\end{equation}
hence the resulting algebras have involutive inversion. 
Below, we shall prove that every real division algebra with involutive inversion is
isomorphic to one of these, and that they are mutually non-isomorphic.

Every pair $\s,\tau$ of automorphisms of a real alternative division algebra $B$
satisfying (\ref{d6}) define a representation of 
$\mathrm{D}_6=\<a,b\mid a^2=b^2=(ab)^3=\I_B\>$ on $B$ via $a\mapsto\s$, $b\mapsto\tau$.
The irreducible real representations of $\mathrm{D}_6$ are the trivial representation, the
sign representation, and the natural representation (viewing $\mathrm{D}_6$ as a subgroup of
$\gl_2(\R)$ in the natural way). We denote these representations by $\triv$, $\sign$ and
$\nat$, respectively.

Setting $T_{\s,\tau}=\ker(\s-\I_B)\cap\ker(\tau-\I_B)$,  
$S_{\s,\tau}=\ker(\s+\I_B)\cap\ker(\tau+\I_B)$, and 
$N_{\s,\tau}=T_{\s,\tau}^\perp\cap S_{\s,\tau}^\perp\subset B$, we have 
$B=T_{\s,\tau}\oplus S_{\s,\tau} \oplus N_{\s,\tau}$ as a $\mathrm{D}_6$-module, and
$T_{\s,\tau}\simeq\dim(T_{\s,\tau})\,\triv$, $S_{\s,\tau}\simeq\dim(S_{\s,\tau})\,\sgn$,
$N_{\s,\tau}\simeq(\dim(N_{\s,\tau})/2)\,\nat$. 
Observe that since $\s,\tau\in\Aut(B)\subset\og(B)$, the above direct sum decomposition is
orthogonal, and $1_B\in T_{\s,\tau}$.

Define $\rho=\tau\s\in\Aut(B)$.
Now $\rho$ acts as the identity on $T_{\s,\tau}\oplus S_{\s,\tau}$, while
$\rho|_{N_{\s,\tau}}$ is a rotation with angle $2\pi/3$:
$\<\rho(x),x\>=\|x\|^2\cos(2\pi/3)$ for every $x\in N_{\s,\tau}$.

\begin{lma} \label{squareidp}
  For all $x\in N_{\s,\tau}$, there exists a unit vector $y\in S_{\s,\tau}$ such that
  $\s(x)\tau(x)=-\|x\|^2(c -sy)$. Moreover, $\s(c-sy)\tau(c-sy)=c-sy$.
\end{lma}

\begin{proof}
We may assume that $\|x\|=1$.
Now $\rho(x)x=-\<\rho(x),x\> + \rho(x)\!\times\! x = -c-sy$ for some unit vector
$y\in\Im(B)$, where '$\times$' denotes the vector product on $\Im(B)$ (\cf{} Section~\ref{leftunital}).
The subspace $\spann\{x,\rho(x)\}\subset B$ is a submodule, isomorphic to the natural
representation of $\mathrm{D}_6$. Hence there exist vectors $u,u',v,v'\in\spann\{x,\rho(x)\}$
such that $\s(u)=u$, $\s(u')=-u'$, $\tau(v)=v$ and $\tau(v')=-v'$, hence 
$\s(u\!\times\! u')=-u\!\times\! u'$ and $\tau(v\!\times\! v')=-v\!\times\! v'$.
But $y= \frac{1}{s}\left(\rho(x)-\<\rho(x),x\>\right)\!\times\! x =\pm u\!\times\! u'=
 \pm v\!\times\! v'$, so
$y\in\ker(\sigma+\I)\cap\ker(\tau+\I)=S_{\s,\tau}$. 

It follows that $\s(x)\tau(x)=\tau\left(\rho(x)x\right)=-c-\tau(y)= -(c-sy)$.
A direct computation gives $\s(c-sy)\tau(c-sy)=c-sy$. 
\end{proof}

\begin{lma} \label{noniso}
  Let $A=\O_{\a,\b}$ and $B=\O_{\s,\tau}$, where $\a,\b,\s,\tau\in\Aut(\O)$.
  Assume that the representations of $\mathrm{D}_6$ given by $\a,\b$ and $\s,\tau$ are isomorphic
  to $2\,\triv\oplus2\,\sgn\oplus2\,\nat$ respectively
  $\triv\oplus\sgn\oplus3\,\nat$.
  Then the algebras $A$ and $B$ are not isomorphic.
\end{lma}

\begin{proof}
Write $A=(\O,\circ)$ and $B=(\O,*)$.
By Lemma~\ref{squareidp}, the square of every element in the 6-dimensional subspace
$N_{\s,\tau}\subset B$ is proportional to an idempotent, and since $\dim S_{\s,\tau}=1$ it
follows that the idempotent is the same for all elements in $N_{\sigma,\tau}$.
Suppose that $U\subset A$ is a subspace of dimension 6 such that the square of every
element is proportional to a fixed idempotent $e$. Then $U$ intersects the subspace
$V=(T_{\a,\b}\oplus S_{\a,\b})\cap\Im\O$ non-trivially, and since
$x^{\circ2}\in\spann\{1_\O\}$ for every $x\in V$, it follows that $e=1_\O$.
On the other hand, it is easy to see that 
$\{x\in A\mid x^{\circ2}\in\spann\{1_\O\}\}=\spann\{1_\O\}\cup V$, so no 6-dimensional
subspace $U$ with the desired property can exist in $A$. 
This means that $A$ and $B$ are not isomorphic. 
\end{proof}

Lemma~\ref{noniso} implies that the algebras specified by (\ref{222}) and (\ref{113}) are
not isomorphic.
In addition, the algebras defined by (\ref{8}) and (\ref{44}) are distinguished from the
ones defined by (\ref{222}) and (\ref{113}) by the existence of a non-zero central
idempotent, while the one defined by (\ref{8}) is unique among the four above-mentioned in
having an identity element. Thus these algebras are mutually non-isomorphic.
Similarly, in the 4-dimensional case, the algebra defined by (\ref{4}) is unital, while
the one defined by (\ref{22}) has a non-zero central idempotent which is not an identity
element.
This shows that the algebras defined in Proposition~\ref{invinv} belong to different
isomorphism classes.

To prove Proposition~\ref{invinv}, it remains to show that every real division algebra
with involutive inversion is isomorphic to one of those specified by
(\ref{1})--(\ref{113}).
First, if $B$ is a real alternative division algebra and $\s\in\Aut(B)$ an automorphism
of order two,
then $\dim(\ker(\s-\I))=\dim(\ker(\s+\I))=\dim(B)/2$.
Consequently, if $\s,\tau\in\Aut(B)$ satisfy (\ref{d6}) then either $\s=\tau=\I_B$ or
$\dim(S_{\s,\tau})=\dim(T_{\s,\tau})\ne0$.
Hence, as a $\mathrm{D}_6$-module, $B$ must be isomorphic to either
\begin{align*}
  &\triv &&&&&\mbox{if }& B\simeq\R, \\
  &2\,\triv, &&&\mbox{or }& \triv\oplus\sgn &\mbox{if }& B\simeq\C, \\
  &4\,\triv,&& 2\,\triv\oplus2\,\sgn, &\mbox{or }& \triv\oplus\sgn\oplus\nat 
    &\mbox{if }& B\simeq\H, \\
  &8\,\triv,&& 4\,\triv\oplus 4\,\sgn,&& 3\,\triv\oplus3\,\sgn\oplus\nat, \\
  &&&2\,\triv\oplus2\,\sgn\oplus2\,\nat,& \mbox{or } & \triv\oplus\sgn\oplus3\,\nat 
  &\mbox{if }& B\simeq\O.
\end{align*}
This immediately gives the result for $\dim B\le2$.

If $B_{\a,\b}$ and $B_{\s,\tau}$ have involutive inversion and $\f\in\Aut(B)$ satisfies 
$(\s,\tau)=(\f\a\f\inv,\f\b\f\inv)$, then $\f:B_{\a,\b}\to B_{\s,\tau}$ is an isomorphism
of algebras as well as of $\mathrm{D}_6$-modules.
Assume that $\s\in\Aut(\O)\setminus\{\I_\O\}$ and $\s^2=\I_\O$. Then there exists a Cayley triple
$(u,v,z)$ in $\O$ such that $u,v\in\ker(\s-\I_\O)$, $z\in\ker(\s+I_\O)$, \ie, there is an
automorphism $\f$ of $\O$ such that $\f\s\f\inv\cdot(i,j,l)=(i,j,-l)$. 
This proves that $\O_{\s,\s}$ is isomorphic to the algebra given by (\ref{44}).
Similarly, one proves that if $\s$ is an automorphism of $\H$ of order two then
$\H_{\s,\s}$ is isomorphic to the algebra specified by (\ref{22}).

From here on, assume that $\s,\tau\in\Aut(B)$ satisfy Equation~(\ref{d6}), and
$\s\ne\tau$. 

\begin{lma}
  If $B\simeq\H$ then $\H_{\s,\tau}$ is isomorphic to the algebra defined by (\ref{111}).
\end{lma}

\begin{proof}
From the assumptions follows that $\H_{\s,\tau}\simeq\triv\oplus\sgn\oplus\nat$ as a
$\mathrm{D}_6$-module. Let $u\in N_{\s,\tau}\cap\ker(\s-\I_\H)$ and 
$v\in N_{\s,\tau}\cap\ker(\s+\I_\H)$.  
Then $\s(u)=u$ and $\s(v)=-v$, while $\tau(u)=cv+sv$, $\tau(v)=-cv+su$. Hence there exists
an automorphism $\f$ of $\H$ such that $\f\s\f\inv\cdot(i,j)=(i,-j)$ and
$\f\tau\f\inv\cdot(i,j)=(ci+sj,si-cj)$, that is, $\H_{\s,\tau}$ is isomorphic to the
algebra given by (\ref{111}).
\end{proof}

\begin{lma}
  If $B\simeq\O$ and the $\mathrm{D}_6$-module $B_{\s,\tau}$ decomposes as
  $\triv\oplus\sgn\oplus3\,\nat$ then $B_{\s,\tau}$ is isomorphic to the algebra given by
  (\ref{113}). 
\end{lma}

\begin{proof}
Let $X=N_{\s,\tau}\cap\ker(\s-\I_\O)$. For all $x\in X$ we have $\s(x)=x$ and
$\tau(x)=cx+sy_x$ for some $y_x\in N_{\s,\tau}\cap\ker(\s+\I_\O)$, and from
Lemma~\ref{squareidp} follows that $xy_x\in S_{\s,\tau}$. 
Now $x\mapsto xy_x$ defines a continuous map from the unit sphere of $X$ to the unit
sphere of $S_{\s,\tau}$, and since $\dim S_{\s,\tau}=1$, this map is constant.
Take $z=-xy_x$, where $x\in X$ is any unit vector. Then $xz=x(-xy_x)=y_x$, hence $\tau(x)=cx+sxz$.

Let $u,v$ be an orthonormal pair in $X$, then $(u,v,z)$ is a Cayley triple satisfying
$\s\cdot(u,v,z)=(u,v,-z)$ and $\tau\cdot(u,v,z)=(cu+suz,cv+svz,-z)$. This concludes the
proof of the lemma. 
\end{proof}

\begin{lma}
  If $B\simeq\O$ and $B_{\s,\tau}\simeq 2\,\triv\oplus2\,\sgn\oplus2\,\nat$ as a
  $\mathrm{D}_6$-module, then the algebra $B_{\s,\tau}$ is isomorphic to the algebra given by (\ref{222}). 
\end{lma}

\begin{proof}
Let $z\in N_{\s,\tau}\cap\ker(\s-\I_\O)$, $\|z\|=1$, then $\tau(z)=cz+sy$ for a unit vector
$y\in N_{\s,\tau}\cap\ker(\s+\I_\O)$. It follows from Lemma~\ref{squareidp} that $yz\in S_{\s,\tau}$, and setting
$u=-yz$ we have $uz=(-yz)z=y$, so $\tau(z)=cz+suz$. Let $v\in S_{\s,\tau}$ be a unit
vector orthogonal to $u$. Then $uv$ is in $T_{\s,\tau}$ and hence orthogonal to $z$, so
$(u,v,z)$ is a Cayley triple, which satisfies $\s\cdot(u,v,z)=(-u,-v,z)$ and
$\tau\cdot(u,v,z)=(-u,-v,cz+suz)$.
\end{proof}

To conclude the proof of Proposition~\ref{invinv}, it remains only to show that the module
$\O_{\s,\tau}$ cannot decompose as $3\,\triv\oplus3\,\sgn\oplus\nat$.
Assume $\dim S_{\s,\tau}\ge3$, and let $u,v,z\in S_{\s,\tau}$ be orthonormal. Then $uv\in
T_{\s,\tau}$, so $(u,v,z)$ is a Cayley triple, and $\{1,u,v,uv,z,uz,vz,(uv)z\}$ is a basis
of $\O$, with $u,v,z,(uv)z\in S_{\s,\tau}$ and $1,uv,uz,vz\in T_{\s,\tau}$, so
$N_{\s,\tau}=0$ and thus $\O_{\s,\tau}\simeq4\,\triv\oplus4\,\sgn$ as a
$\mathrm{D}_6$-module.

\section{Quasigroup identites} \label{sec:qgi}

One basic result in group theory is that for any group $G$ the maps $y \mapsto xyx^{-1}$,
$x \in G$ are automorphisms. This result is a consequence of the fact that the associative law
$(xy)z =x(yz)$ holds in groups.  {Since partial classifications of real division algebras with
large groups of automorphisms have been obtained \cite{BO81a,BO81b,DbZh}, any method for translating
quasigroup identities to results about the existence of automorphisms on the underlying
real division algebra will produce classification results for division algebras.} Unfortunately these methods
are scarce, so at the present time this approach seems unfruitful. A closer look at some
particular varieties of 
quasigroups makes clear that sometimes it is easy to translate quasigroup identities to
results about the existence of \emph{autotopies}, \ie. isotopies from the quasigroup to
itself. For instance, the associative law holds in a quasigroup if and only if
$(\rt_z,\I,\rt_z)$ and $(\lt_x,\lt_x,\I)$ are autotopies; the left Moufang identity
$x(y(xz)) = ((xy)x)z$ is equivalent to $(\lt_x,\rt_x\lt_x,\lt^{-1}_x)$ being an autotopy, etc.
The set of autotopies of an algebra $(A,xy)$ forms a group under componentwise composition, the
\emph{autotopy group} of $A$, denoted by $\Atp(A)$, \ie,
\begin{equation}\label{eq:Atp}
  \Atp(A) =  \{(\varphi_1,\varphi_2,\varphi_3) \in \End(A)^3 \mid \varphi_1(xy) = \varphi_2(x)\varphi_3(y)\: \forall_{x,y \in A}\}.
\end{equation}
The vector space
\begin{displaymath}
    \Tder(A) = \{(d_1,d_2,d_3) \in \End_k(A)^3 \mid d_1(xy) = d_2(x) y + x d_3(y) \: \forall_{x,y \in A}\}
\end{displaymath}
is a Lie algebra with the componentwise commutator of linear maps. The elements of $\Tder(A)$ are called \emph{ternary derivations}. In case that $(A,xy)$ is a real algebra, $\Atp(A)$ is a Lie group and $\Tder(A)$ is the Lie algebra of $\Atp(A)$ at the identity.

\begin{lma}
\label{lem:tangent}
Let $(A,xy)$ be an algebra over the real numbers, $\epsilon > 0$ and
$(-\epsilon, \epsilon) \to \Atp(A)$ $t \mapsto (\alpha_t,\beta_t,\gamma_t)$ a
differentiable curve with $\alpha_0 = \beta_0 = \gamma_0 = \I_{\Atp(A)}$. Then
$(\dot{\alpha}_0,\dot{\beta}_0,\dot{\gamma}_0) \in \Tder(A)$.
\end{lma}
\begin{proof}
  It suffices to compute the derivative of  $\alpha_t(xy) = \beta_t(x)\gamma_t(y)$ at $t = 0$.
\end{proof}

In \cite{JP08}, a construction was given of all real division algebras whose Lie algebra
of ternary derivations has a simple subalgebra of toral rank 2.
Quasigroup identities in division algebras tend to imply existence of many autotopies,
potentially enough for the autotopy group to contain a Lie subgroup of toral rank 2.
This leads us to approach real division algebras satisfying quasigroup identities through
their autotopy group, seeking to obtain general results about these algebras.

The \emph{left, middle and right associative nuclei} of a $k$-algebra $(A,xy)$ are
\begin{eqnarray*}
  \N_l(A) &=& \{ a \in A \mid (ay)z = a(yz) \: \forall_{y,z \in A}\} \,,\\
  \N_m(A)  &=& \{ a \in A \mid (xa)z = x(az) \: \forall_{x,z \in A}\} \,,\\
  \N_r(A) &=& \{ a \in A \mid (xy)a = x(ya) \: \forall_{x,y \in A}\}
\end{eqnarray*}
respectively. These nuclei are associative subalgebras of $A$, so for 
real division algebras they are isomorphic to either $0, \R, \C$ or $\H$.
If $A$ is an algebra with unit element, then
\begin{eqnarray*}
\{(d_1,d_2,d_3) \in \Tder(A) \mid d_1 = 0\} &=& \{(0,\rt_a,-\lt_a) \mid a \in \N_m(A)\} \quad \text{and}\\
\{(d_1,d_2,d_3) \in \Tder(A) \mid d_2 = 0\} &=& \{(\rt_a,0,\rt_a) \mid a \in \N_r(A)\},\\
 \{(d_1,d_2,d_3) \in \Tder(A) \mid d_3 = 0\} &=& \{(\lt_a,\lt_a,0) \mid a \in \N_l(A)\}.
\end{eqnarray*}

The maps
\begin{eqnarray*}
 \pi_i \colon \Tder(A) & \to & \End_k(A)\\
 (d_1,d_2,d_3) & \mapsto & d_i
\end{eqnarray*}
provide representations of the Lie algebra $\Tder(A)$ on $A$. So, depending on the
representation $\pi_1$, $\pi_2$ or $\pi_3$ that we choose, $A$ inherits three structures of
$\Tder(A)$-module, which we denote by $A_1, A_2$ and $A_3$ respectively. The product on
$A$ is a homomorphism 
\begin{displaymath}
A_2 \otimes A_3 \to A_1
\end{displaymath}
of $\Tder(A)$-modules. The image of $\Tder(A)$ under $\pi_i$  will be denoted by
$\Tder(A)_i$, and $\Atp(A)_i=p_i(\Atp(A))$ for
$p_i:\Atp(A)\to\End_k(A),\:(\f_1,\f_2,\f_3)\mapsto\f_i$.

\begin{prop}
\label{prop:PI}
 Let $(A,xy)$ be a real division algebra, $e \in A$ a non-zero
 idempotent and $x*y = (x/e)(e\backslash y)$. Let $f, g \in \gl(A)$ and let $\alpha,
 \beta, \gamma \colon U \to \gl(A)$ $x \mapsto \alpha_x, \beta_x, \gamma_x$ be maps from a
 neighbourhood $U$ of $e$ in $A$ to $\gl(A)$ that are differentiable at $e$.
 Given $S \in \{\lt,\rt\}$ and $\epsilon \in \{1, -1\}$:
\begin{enumerate}
\item if $f(S^\epsilon_x g(yz)) = \beta_x(y)\gamma_x(z)$ for all $x\in U$ and  $y,z \in A$ then there exists $h \in \gl(A)$ such that $h S^*_{x} h^{-1} \in \Tder(A,*)_1$ for all $x \in A$;
\item if $\alpha_x(yz) = f(S^\epsilon_xg(y))\gamma_x(z)$ for all $x\in U$ and  $y,z \in A$ then there exists $h \in \gl(A)$ such that $h S^*_{x} h^{-1} \in \Tder(A,*)_2$ for all $x \in A$;
\item if $\alpha_x(yz) = \beta_x(y)f(S^\epsilon_xg(z))$ for all $x\in U$ and  $y,z \in A$ then there exists $h \in \gl(A)$ such that $h S^*_{x} h^{-1} \in \Tder(A,*)_3$ for all $x \in A$.
\end{enumerate}
\end{prop}
\begin{proof}
We will only prove the case (2) with $S=\lt$, $\e=-1$, that is,
$\alpha_x(yz) = f(\lt^{-1}_x g(y))\gamma_x(z)$. The proofs of the other statements are
similar.

First observe that $(\alpha_x, f \lt^{-1}_x g, \gamma_x) \in \Atp(A)$ implies that
\[
(\alpha^{-1}_e\alpha_x, g^{-1}\lt_e \lt^{-1}_x g, \gamma^{-1}_e\gamma_x)
\]
also belongs to $\Atp(A)$. The multiplication operator $\lt^*_x$ is $\lt_{x/e}\lt^{-1}_e$
so $g^{-1}\lt^{*}_x g = ( g^{-1}\lt_e \lt^{-1}_{x/e} g)^{-1} \in \Atp(A)_2$ for any $x$ in
a certain neighborhood of $e$. The derivative of  the curve $g^{-1}\lt^{*}_{e + t y} g$ at
$t = 0$ is $g^{-1}\lt^*_{y}g$. Thus, for any $y$ in a neighborhood of $0$ we have that
$g^{-1}\lt^*_{y}g \in \Tder(A)_2$. Since $\Tder(A,*)_2 = \rt_e \Tder(A)_2 \rt^{-1}_e$, the
result follows.
\end{proof}

Since $(A,*)$ is a division algebra, there are no proper non-zero invariant subspaces of
$A$ under the action of $\{ h S^*_{x} h^{-1} \mid x \in A \}$ ($S \in \{ \lt,\rt\}$).
Therefore, Proposition \ref{prop:PI} establishes that under certain generic
hypotheses $\Tder(A,*)_i$ is not too small and it has a rich structure as a Lie
algebra. In \cite{JP08} it was proved that the largest possibilities for $\Tder(A)$ only occur
for isotopes of Hurwitz algebras.

\begin{thm}
\label{thm:PI}
Let $(A,*)$ be a unital real division algebra. If there exist $h \in
\gl(A)$ and $\iota \in \{1,2,3\}$ such that $\{ h \lt^*_x h^{-1} \mid x \in A \}$ or $\{ h
\rt^*_x h^{-1} \mid x \in A\}$ is contained in $\Tder(A,*)_\iota$ then $(A,*)$ is a
Hurwitz algebra.
\end{thm}

\begin{cor}
\label{cor:PI}
Under the hypotheses of Proposition \ref{prop:PI} the algebra $(A,*)$ with product $x*y = (x/e)(e \backslash y)$ is a Hurwitz algebra. If in addition $(\lt_e, \rt^{-1}_e\lt_e\rt_e,\lt_e)$ (resp.\ $(\rt_e, \rt_e, \lt_e^{-1}\rt_e\lt_e)$) belongs to $\Atp(A,xy)$ then $\lt_e$ (resp.\ $\rt_e$) is an automorphism of $(A, x*y)$.
\end{cor}

A proof of Theorem~\ref{thm:PI} is given in Section~\ref{sec:proof}.
Unfortunately it is quite technical, using most of the results in \cite{JP08}.
It would be desirable to find a more straightforward proof of this result.

\section{Examples} \label{sec:examples}

In this section we present several examples illustrating how Corollary \ref{cor:PI} can be
used to classify real division algebras satisfying certain quasigroup
identities. While we can quickly determine whether the algebras of certain varieties are isotopic
to Hurwitz algebras, the classification of the isotopy maps is often quite laborious.

\subsection{The identity $x((xy)(xz)) = (x^2(yx))z$}

As a first example we will classify all real division algebras that
satisfy the identity $x((xy)(xz)) = (x^2(yx))z$, labelled with the number 19 in Table
\ref{Tb:d5}. 
We have chosen this identity since it is one of the most difficult ones in
Table~\ref{Tb:d5} to deal with.
\begin{prop}
\label{prop:I1}
A real division algebra $(A,xy)$ satisfies the identity
\begin{equation}   \label{eq:example_main_theorem_2}
 x((xy)(xz)) = (x^2(yx))z
\end{equation}
if and only if it has involutive inversion.
\end{prop}
\begin{proof}
Remember that a real division algebra $A$ has involutive inversion if and only if
the product $xy$ of $A$ can be expressed as $xy = \sigma(x)*\tau(y)$, where
$(A,*)$ is isomorphic to $\R$, $\C$, $\H$ or $\O$, and $\sigma, \tau$ are automorphisms of
$(A,*)$ satisfying
$\sigma^2 = \I_A = \tau^2$ and $\sigma \tau \sigma = \tau \sigma \tau$.

Assume that $A$ is a real division algebra satisfying
Equation~(\ref{eq:example_main_theorem_2}), and consider the algebra $(A,*)$ with 
$x*y = (x/e)(e\backslash y)$, where $e$ is a non-zero idempotent of $A$. 
Since $x(yz) = (x^2((x\backslash y)x))(x\backslash z)$, Corollary~\ref{cor:PI}
tells us that  $(A,*)$ is a Hurwitz algebra. Equation (\ref{eq:example_main_theorem_2})
evaluated at $x = y = e$ implies that $\lt^2_e = \I_A$. The same equation evaluated at $x
= z = e$ gives $\lt_e\rt_e\lt^{-1}_e = \rt^{-1}_e\lt_e\rt_e$ and, by Lemma
\ref{lem:group}, $\rt^2_e = \I_A$. Since $(\lt_e, \lt_e\rt_e\lt^{-1}_e,\lt^{-1}_e) \in
\Atp(A)$, Corollary \ref{cor:PI} implies that $\tau = \lt_e$ is an automorphism of
$(A,*)$. Set $\sigma = \rt_e$. To conclude the proof of Proposition~\ref{prop:I1} we
need to prove that $\sigma$ is also an automorphism of $(A,*)$.

Equation (\ref{eq:example_main_theorem_2}) is equivalent to
\[
\sigma(x)*(\tau\sigma(\sigma(x)*\tau(y))*(\sigma(x)*\tau(z)))= \sigma(\sigma(\sigma(x)*\tau(x))*(\tau\sigma(y)*x))*\tau(z).
\]
Using the left Moufang identity on $(A,*)$ and replacing $x$ with $\sigma(x)$, this latter equation is equivalent to
\begin{equation}
 \label{eq:example_main_theorem_2_1}
x*\tau\sigma(x*\tau(y))*x = \sigma(\sigma(x*\tau\sigma(x)) *(\tau\sigma(y)*\sigma(x))).
\end{equation}
Linearizing this equation and evaluating at $e$ gives
\begin{eqnarray*}
&& w*\tau\sigma\tau(y) + \tau\sigma(w*\tau(y)) + \tau\sigma\tau(y)*w =\\
 && \quad \sigma\left(\rule{0pt}{2ex}\sigma(w)*\tau\sigma(y) + \sigma\tau\sigma(w)*\tau\sigma(y) + \tau\sigma(y)*\sigma(w)\right).
\end{eqnarray*}
Since $(A,*)$ is a Hurwitz algebra, we have $x*y + y*x = t(x)y + t(y)x -2(x,y)e$, hence
\begin{eqnarray*}
&& t(\tau\sigma\tau(y))w + t(w)\tau\sigma\tau(y) - 2 (w,\tau\sigma\tau(y))e + \tau\sigma(w*\tau(y)) =\\
&& \quad  t(\tau\sigma(y))w + t(\sigma(w))\sigma\tau\sigma(y) - 2 (\tau\sigma(y),\sigma(w))e + \sigma(\sigma\tau\sigma(w)*\tau\sigma(y))
\end{eqnarray*}
which, substituting $y$ for $\tau(y)$ and applying $\sigma$ to each side gives
\begin{eqnarray}
\label{eq:example_main_theorem_2_2}
&& t(\tau\sigma(y))\sigma(w) + t(w)\sigma\tau\sigma(y) - 2 (w,\tau\sigma(y))e + \sigma\tau\sigma(w*y) =\\
&& \quad t(\tau\sigma\tau(y))\sigma(w) + t(\sigma(w))\tau\sigma\tau(y) - 2 (\tau\sigma\tau(y),\sigma(w))e + \sigma\tau\sigma(w)*\tau\sigma\tau(y).\nonumber
\end{eqnarray}
If $\sigma$ is an isometry with respect to $(\,,\,)$ then the previous equality implies that $\tau\sigma\tau$ is an automorphism of $(A,*)$ and $\sigma$ is an automorphism of $(A,*)$ as desired.

The identity (\ref{eq:example_main_theorem_2_2}) with $y = w$ gives, after replacing $w$
with $\sigma(w)$, 
\begin{eqnarray}
\label{eq:example_main_theorem_2_3}
 && \left(\rule{0pt}{2ex}2t(\sigma(w)) - t(w) - t(\sigma\tau(w))\right)\sigma\tau(w) + \left(\rule{0pt}{2ex}t(w) -t(\sigma\tau(w))\right)w +\\
&& \quad \left(\rule{0pt}{2ex} 2(w,\sigma\tau(w)) + n(\sigma\tau(w)) - 2 (\sigma(w),\tau(w)) - n(\sigma(w))\right)e = 0. \nonumber
\end{eqnarray}
If $2t(\sigma(w)) - t(w) - t(\sigma\tau(w)) \neq 0$ then $\sigma\tau(w) \in \R e + \R
w$. {Since $\sigma\tau(e) = e$, the minimal polynomial of the restriction of
  $\sigma\tau$ to $\R e + \R w$ is of the form $(x-1)(x-\lambda)^\epsilon$ for some
  $\epsilon \in \{0,1\}$ and $\lambda \in \R$. This polynomial also divides $x^3 -1 =
  (x-1)(x^2 +x +1)$ because $(\sigma\tau)^3 = \I_A$, so $\epsilon=0$ and $\sigma\tau(w) = w$ and  $\sigma(w) = \tau(w)$}. The set $\{ w \in A \mid 2t(\sigma(w)) - t(w) -
t(\sigma\tau(w)) \neq 0\}$ is either empty or an open dense set in the Zariski topology of
$A$, so $2t(\sigma(w)) - t(w) - t(\sigma\tau(w)) \neq 0$ implies that $\sigma =
\tau$. Therefore we may assume that $2t(\sigma(w)) - t(w) - t(\sigma\tau(w)) = 0$ for all
$w\in A$. A similar argument with the coefficient of $w$ in
(\ref{eq:example_main_theorem_2_3}) shows that $t(w) -t(\sigma\tau(w)) = 0$. In particular
$t(\sigma(w)) = t(w)$. With this new information (\ref{eq:example_main_theorem_2_2}) can
be written as
\[
 \tau\sigma\tau(w*y) = \tau\sigma\tau(w)* \tau\sigma\tau(y) + 2(w,\tau\sigma(y)) e - 2 (\tau\sigma\tau(y),\sigma(w)) e
\]
which, after substituting $w$ and $y$ for $\tau(w)$ and $\tau(y)$ respectively and
applying $\tau$, becomes
\[
 \sigma(w*y) = \sigma(w)*\sigma(y) + 2(w,\sigma\tau(y))e - 2(\tau\sigma\tau(w),\sigma(y))e.
\]
Using this formula, Equation~(\ref{eq:example_main_theorem_2_1}) and the middle Moufang identity imply that either $(\sigma\tau(x),\sigma\tau(y))= (x,y)$ or $x*\tau\sigma(x) \in \R e + \R x$. If $(\sigma\tau(x),\sigma\tau(y))= (x,y)$ then $\sigma$ is an isometry and we are done. Otherwise, $\tau\sigma(x) \in \R e + \R x$ for $x$ in a Zariski dense set of $A$. {As above, this implies that $\sigma(x) = \tau(x)$ for $x$ in a dense set. Therefore, $\sigma = \tau$ is  an automorphism of $(A,*)$.}
\end{proof}

\subsection{Bol-Moufang type identities}
A quasigroup identity is of \emph{Bol-Moufang type} if two of its three variables occur once on each side, the third variable occurs twice on each side, the order in which the variables appear on both sides is the same and the only binary operation used is the multiplication. In \cite{PV05} J. D. Phillips and P. Vojt{\v{e}}chovsk{\'y} proved that there are exactly 26 varieties of quasigroups of Bol-Moufang type and they determined all inclusions between the varieties. In Figure \ref{Fg:QTree} varieties in an upper level connected with varieties in a lower level indicate that the former are included in the latter.  The superscript immediately following the abbreviation of the name of the variety $\mathcal{A}$ indicates if:
\begin{itemize}
  \item[(2)] every quasigroup in $\mathcal{A}$  is a loop,
  \item[(L)] every quasigroup in $\mathcal{A}$ is a left loop, and there is $Q \in \mathcal{A}$ that is not a loop.
  \item[(R)] every quasigroup in $\mathcal{A}$ is a right loop, and there is $Q \in \mathcal{A}$ that is not a loop,
  \item[(0)] there is $Q_L \in \mathcal{A}$ that is not a left loop, and there is $Q_R \in \mathcal{A}$ that is not a right loop.
\end{itemize}
For instance, quasigroups in the variety MNQ are loops. 
Each of the varieties can be defined by a single identity among a set of equivalent identities. In
Table~\ref{Tb:varieties}, one defining identity has been chosen for each variety. The reader
will find detailed information in \cite{PV05}. 

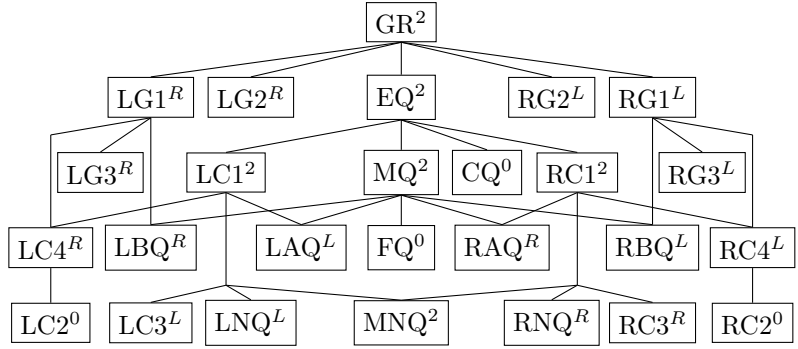
\begin{figure}[ht]
\centering
\begin{tikzpicture}
\path   node (raiz)at ( 4.66,0) [shape=rectangle,draw] {GR$^2$}
        node (11) at ( 1.33,-1) [shape=rectangle,draw] {LG1$^R$}
        node (12) at ( 2.66,-1) [shape=rectangle,draw] {LG2$^R$}
        node (13) at ( 4.66,-1) [shape=rectangle,draw] {EQ$^2$}
        node (14) at ( 6.66,-1) [shape=rectangle,draw] {RG2$^L$}
        node (15) at ( 8,-1) [shape=rectangle,draw] {RG1$^L$}
        node (21) at ( 0.66,-2) [shape=rectangle,draw] {LG3$^R$}
        node (22) at ( 2.33,-2) [shape=rectangle,draw] {LC1$^2$}
        node (23) at ( 4.66,-2) [shape=rectangle,draw] {MQ$^2$}
        node (24) at ( 5.8,-2) [shape=rectangle,draw] {CQ$^0$}
        node (25) at ( 7,-2) [shape=rectangle,draw] {RC1$^2$}
        node (26) at ( 8.66,-2) [shape=rectangle,draw] {RG3$^L$}
        node (31) at ( 0,-3) [shape=rectangle,draw] {LC4$^R$}
        node (32) at ( 1.33,-3) [shape=rectangle,draw] {LBQ$^R$}
        node (33) at ( 3.33,-3) [shape=rectangle,draw] {LAQ$^L$}
        node (34) at ( 4.66,-3) [shape=rectangle,draw] {FQ$^0$}
        node (35) at ( 6,-3) [shape=rectangle,draw] {RAQ$^R$}
        node (36) at ( 8,-3) [shape=rectangle,draw] {RBQ$^L$}
        node (37) at ( 9.33,-3) [shape=rectangle,draw] {RC4$^L$}
        node (41) at ( 0,-4) [shape=rectangle,draw] {LC2$^0$}
        node (42) at ( 1.33,-4) [shape=rectangle,draw] {LC3$^L$}
        node (43) at ( 2.66,-4) [shape=rectangle,draw] {LNQ$^L$}
        node (44) at ( 4.66,-4) [shape=rectangle,draw] {MNQ$^2$}
        node (45) at ( 6.66,-4) [shape=rectangle,draw] {RNQ$^R$}
        node (46) at ( 8,-4) [shape=rectangle,draw] {RC3$^R$}
        node (47) at ( 9.33,-4) [shape=rectangle,draw] {RC2$^0$};
\draw (raiz.south)--(11.north);
\draw (raiz.south)--(12.north);
\draw (raiz.south)--(13.north);
\draw (raiz.south)--(14.north);
\draw (raiz.south)--(15.north);
\draw (11.south)--(0,-1.5);
\draw (0,-1.5)--(31.north);
\draw (11.south)--(21.north);
\draw (11.south)--(32.north);
\draw (13.south)--(22.north);
\draw (13.south)--(23.north);
\draw (13.south)--(24.north);
\draw (13.south)--(25.north);
\draw (15.south)--(36.north);
\draw (15.south)--(26.north);
\draw (15.south)--(9.33,-1.5);
\draw (9.33,-1.5)--(37.north);
\draw (22.south)--(31.north);
\draw (22.south)--(2.33,-3.5);
\draw (22.south)--(33.north);
\draw (23.south)--(32.north);
\draw (23.south)--(33.north);
\draw (23.south)--(34.north);
\draw (23.south)--(35.north);
\draw (23.south)--(36.north);
\draw (25.south)--(35.north);
\draw (25.south)--(7,-3.5);
\draw (25.south)--(37.north);
\draw (31.south)--(41.north);
\draw (2.33,-3.5)--(42.north);
\draw (2.33,-3.5)--(43.north);
\draw (2.33,-3.5)--(44.north);
\draw (7,-3.5)--(44.north);
\draw (7,-3.5)--(45.north);
\draw (7,-3.5)--(46.north);
\draw (37.south)--(47.north);
\end{tikzpicture}
\caption{Varieties of quasigroups of Bol-Moufang type}
\label{Fg:QTree}
\end{figure}

\begin{table}
\caption{Defining identities.}
\centering\begin{tabular}{|c|c|c|c|c|c|}
  \hline
    \scriptsize CQ &  \scriptsize  $x(y(yz)) = ((xy)y)z$ &  
    \scriptsize EQ &  \scriptsize  $x((yx)z) = (xy)(xz)$ &
    \scriptsize FQ &  \scriptsize $(x(yx))z = ((xy)x)z$ \\  
    \scriptsize GR &  \scriptsize $(xy)z=x(yz)$ & 
    \scriptsize LAQ &  \scriptsize $x(x(yz)) = (xx)(yz)$ &
    \scriptsize LBQ &  \scriptsize $x(y(xz)) = (x(yx))z$ \\  \scriptsize LC1 &  \scriptsize $(xx)(yz) = (x(xy))z$ &
    \scriptsize LC2 &  \scriptsize $x(x(yz)) = (x(xy))z$ &  \scriptsize LC3 &  \scriptsize $x(x(yz)) = ((xx)y)z$ \\
    \scriptsize LC4 &  \scriptsize $x(y(yz)) = (x(yy))z$ & 
    \scriptsize LG1 & \scriptsize $x(y(zx))=(x(yz))x$ &  
    \scriptsize LG2 & \scriptsize $(xy)(zz)=((x(yz))z$ \\
    \scriptsize LG3 &  \scriptsize  $x(y(zy)) = (x(yz))y$ &  
    \scriptsize LNQ &  \scriptsize $(xx)(yz) = ((xx)y)z$ &  \scriptsize MNQ &  \scriptsize $x((yy)z) = (x(yy))z$ \\  \scriptsize MQ &  \scriptsize $x(y(xz)) = ((xy)x)z$ &
    \scriptsize RAQ &  \scriptsize $(x(yy))z = ((xy)y)z$ &  \scriptsize RBQ &  \scriptsize $x((yz)y) = ((xy)z)y$ \\
    \scriptsize RC1 &  \scriptsize $x(y(zz)) = (xy)(zz)$ &
    \scriptsize RC2 &  \scriptsize $x((yz)z) = ((xy)z)z$ &
    \scriptsize RC3 &  \scriptsize $x(y(zz)) = ((xy)z)z$ \\  
    \scriptsize RC4 &  \scriptsize $x((yy)z) = ((xy)y)z$ &  
    \scriptsize RG1 &  \scriptsize $x((xy)z) = ((xx)y)z$ &
    \scriptsize RG2 &  \scriptsize  $x((xy)z) = (xx)(yz)$ \\ 
    \scriptsize RG3 &  \scriptsize $x((yx)z) = ((xy)x)z$ & \scriptsize RNQ &  \scriptsize $x(y(zz)) = (xy)(zz)$ & 
                    &                                     \\
  \hline
\end{tabular}
\label{Tb:varieties}
\end{table}

We will say that a division algebra $(A,xy)$ is of a certain type X if the quasigroup
$A\setminus \{0\}$ belongs to the variety X. A quick inspection using Corollary \ref{cor:PI}
shows that real division algebras of type LG1, LG2, EQ, RG2, RG1, LG3,
MQ, RG3, LBQ and RBQ are isotopic to Hurwitz algebras. We can strengthen this observation
by a deeper analysis of all the varieties. As we will see most of the examples are
isotopic to Hurwitz algebras.

\begin{prop}
  Any real division algebra of type LC3, LC1, EQ, MNQ, RC3 or RC1 is associative.
\end{prop}
\begin{proof}
Since the opposite algebra of any algebra of type RC3 is an algebra of type LC3, it is enough to prove the statement for algebras of type LC3 and MNQ.

If $(A,xy)$ is of type MNQ then $A$ has a two-sided unit element $e$. Linearizing the identity $x((yy)z) = (x(yy))z$ we obtain that $x((wy+yw)z) = (x(yw+wy))z$. Setting $w = e$ we get the associativity.

Let $(A,xy)$ be of type LC3 and $e$ its left unit element. Linearizing $x(x(yz)) = ((xx)y)z$ we obtain that $x(w(yz)) + w(x(yz)) = ((xw + wx)y)z$. Setting $x = y = e$ we get that $(\rt_e-\I_A)(\rt_e + 2\I_A)=0$. If $\rt_e \neq \I_A$ then there exists an eigenvector $x$ of $\rt_e$ with corresponding eigenvalue $-2$. The identity $x(x(yz)) = ((xx)y)z$ with $y = z = e$ gives that $x^2$ is an eigenvector of eigenvalue $-2$ of $\rt^2_e$. Since the possible eigenvalues of $\rt^2_e$ are $1$ and $4$ we reach a contradiction. This proves that $\rt_e = \I_A$, so $e$ is a two-sided unit element. Setting $w = e$ in $x(w(yz)) + w(x(yz)) = ((xw + wx)y)z$ we get the associativity.
\end{proof}

Real division algebras of type LAQ, RAQ or MQ are alternative, so they are isomorphic to
$\R, \C, \H$ or $\O$. Those of type FQ are flexible and they have been classified
\cite{bbo82,cdkr99,coll,nform}. The classification of the algebras of type LBQ or RBQ
appears in \cite{cu}.

\begin{prop}
  A real division algebra $(A,xy)$ has type LNQ (resp.\ RNQ) if and only if it is either associative or there exist a product $x\circ y$ on $A$ such that $(A,x\circ y)$ is a quadratic division algebra and $xy = (t(x)e-x)\circ y$ (resp.\ $xy = x\circ (t(y)e-y)$) with $x\circ x-t(x)x+n(x)e = 0$ the quadratic equation satisfied by $(A,x\circ y)$.
\end{prop}
\begin{proof}
Let $(A,xy)$ be of type LNQ and let $e$ be its left unit element. The identity $x^2(yz) = (x^2y)z$ implies that $x^2$ belongs to $\N_l(A)$. Thus, $xe + x$ also belongs to $\N_l(A)$. Since $\N_l(A)$ is an associative subalgebra and $e \in \N_l(A)$ then $e$ is the unit element of $\N_l(A)$. It follows that $(xe + x) = (xe + x)e$ so $\rt^2_e = \I_A$.

We distinguish two cases. First we assume that $\N_l(A) = \R e$. In this case $x^2 = n(x)e$ and $xe +x = t(x)e$ for some quadratic form $n(\,)$ and a linear form $t(\,)$. Define the new product $x\circ y = (xe)y$. We have that $x\circ x = (xe)x = (t(x)e-x)x = t(x)x -n(x)e$, so $(A,x\circ y)$ is a quadratic division algebra with unit element $e$ and the products $x\circ y$ and $xy$ are related as required.

The second case deals with the possibility $\dim \N_l(A) \geq 2$. If $\dim A = 8$ and $\dim \N_l(A) = 2$ then the map $A \to \End_k(\N_l(A))$, $x \mapsto (\lt_x + \rt_x)\vert_{\N_l(A)}$ is not injective, so there exists $0 \neq x \in A$ that satisfies $ax + xa = 0$ for all $a \in \N_l(A)$. Choosing $b\in \N_l(A)$ with $b^2 = - x^2$ we obtain that $(b+x)^2 = 0$. Since $A$ is a division algebra, this implies that $x = -b \in \N_l(A)$. However, the only element $x \in \N_l(A)$ such that $xa + ax = 0$ for all $a \in \N_l(A)$ is $x = 0$ (take $a = e$ the unit element), a contradiction. If $\dim A = 8$ and $\dim \N_l(A) = 4$ then $A$ is a left $\N_l(A)$-module, the action given by left multiplication. As a module $A$ decomposes as $A = \N_l(A) \oplus \N_l(A)x$ for some $x \in A$. However, for any $a, b \in \N_l(A)$, the element $(ax)b + b(ax) = a(xb + bx) - (ab)x + (ba)x$ belongs to $\N_l(A)$ so $ab-ba = 0$, a contradiction. We are left with the case in which $\dim A = 4$ and $\dim \N_l(A) = 2$. While an algebraic proof is possible, we prefer a topological approach. We extend the quadratic form on $\N_l(A)$ to a definite positive quadratic form $n(\,)$ on $A$. Thus we have the inclusion of spheres $\iota \colon S^1 = \{x \in \N_l(A) \mid n(x) = 1\} \to S^3 = \{x \in A \mid n(x) = 1\}$. The identity $x^2(yz) = (x^2y)z$ implies the existence of a continuous map $\varphi \colon S^3 \to S^1$ given by $x \mapsto x^2 / \sqrt{n(x^2)}$. The composition $\varphi \circ \iota$ induces the map $(\varphi\circ \iota)_*  = 2\I_{\pi_1(S^1)}$ on the fundamental group $\pi_1(S^1) \cong \Z$ (base point at $e$). However, $(\varphi\circ \iota)_*  = \varphi_* \circ \iota_*$ and $\pi_1(S^3) = 0$, a contradiction.
\end{proof}

\begin{prop}
Any  real division algebra $(A,xy)$ of type LG2 (resp.\ RG2) is either associative, or
there exists a product $x*y$ on $A$ such that $(A,*)$ is isomorphic to the complex numbers,
and $xy=x*\bar{y}$ (resp.{} $xy=\bar{x}*y$), where $x\mapsto\bar{x}$ is the standard
involution on $(A,*)$.
\end{prop}
\begin{proof}
Although we could easily avoid the use of Corollary \ref{cor:PI} in this proof, we will
not. First we linearize $(xy)z^2 =(x(yz))z$ to obtain 
$(xy)(wz+zw) = (x(yw))z + (x(yz))w$. With $w = e$, the right unit element, we get that
$(xy)(ez) = x(yz)$. Corollary~\ref{cor:PI} implies that $xy = x*\sigma(y)$ where $(A,*)$ is
a Hurwitz algebra and $\sigma = \lt_e$ is an automorphism with $\sigma^2 = \I_A$.

The identity $(xy)(ez) = x(yz)$ implies that $(A,*)$ is associative. 
The identity $(xy)z^2 = (x(yz))z$ implies that $\sigma(z)*z = z * \sigma(z)$ for all 
$z\in A$. For a non-zero $z\in\Im(A,*)$, the only elements in $\Im A$ that commute with
$z$ are the ones in $\R z$. Thus, for any $z \in \Im A$ there exists $\lambda_z \in \R$
such that $\sigma(z) = 
\lambda_z z$. It follows that $\sigma$ is either the identity map or $(A,*)$ is isomorphic
to the complex numbers and $\sigma$ is the complex conjugation. 
\end{proof}

\begin{prop}
  A real division algebra is of type CQ if and only if it is either associative, or the
  product is expressible as $xy=\bar{x}*\bar{y}$, where $(A,*)$ is isomorphic to $\C$.
\end{prop}
\begin{proof}
Let $(A,xy)$ be a real division algebra of type CQ.
The defining identity $x(y(yz)) = ((xy)y)z$ implies that $\lt_x\lt_y\lt_y = \lt_{(xy)y}$
and $\rt_z\rt_y\rt_y = \rt_{y(yz)}$. Substituting $(y/y)/y$ for $x$ in 
$\lt_x\lt_y\lt_y = \lt_{(xy)y}$  gives that $A$ has inversion on the left. Substituting
$y\backslash(y\backslash y))$ for $z$  in $\rt_z\rt_y\rt_y = \rt_{y(yz)}$ gives that $A$ has
inversion on the right. Thus, the product $xy$ can be expressed as $xy =
\tau(x)*\sigma(y)$ for some automorphisms $\sigma,\tau$ of the Hurwitz algebra $(A,*)$
with $\sigma^2 = \tau^2 = \I_A$. The identity $x(y(yz)) = ((xy)y)z$ is now equivalent to 
\begin{equation}
\label{eq:CQaux1}
  x*(\sigma\tau(y)*(\tau(y)*z)) = ((x*\sigma(y))*\tau\sigma(y))*z.
\end{equation} 
Linearizing this identity at $y = e$, the unit element of $(A,*)$, gives
$x*((\sigma\tau(y) + \tau(y))*z) = (x*(\sigma(y) + \tau\sigma(y)))*z$. In case that $\dim
A = 8$ this would imply that $\sigma\tau(y) + \tau(y) \in \R e$ for any $y \in
A$. Substituting $y$ for $\tau(y)$ we get that the eigenspace of $\sigma$ corresponding to
the eigenvalue $-1$ has dimension $\geq 7$, a contradiction (the product by an eigenvector
of eigenvalue $-1$ defines a linear isomorphism between the eigenspaces of $\sigma$ of
eigenvalues $1$ and $-1$). Thus $(A,*)$ is associative and (\ref{eq:CQaux1}) is equivalent
to $\sigma\tau(y)*\tau(y) = \sigma(y) * \tau\sigma(y)$. This identity implies that
$\sigma(y) = y$ if and only if $\tau(y) = y$ so $\sigma = \tau$ and $y*\sigma(y) =
\sigma(y)*y$ for any $y \in A$. In particular, $(y + \sigma(y))*(y-\sigma(y)) = y^{*2} -
\sigma(y^{*2}) = 0$ for any $y\in \Im A$. So $\sigma$ acts as $\I$ or $-\I$ on $\Im A$ and
therefore it is either the identity or the standard involution. Since $\sigma$ is an
automorphism this proves that either $\sigma = \I$ or $(A,*)$ is isomorphic to the complex
numbers and $\sigma$ is the conjugation. 
\end{proof}

\begin{prop}
A  real division algebra $(A,xy)$ is of type LG3 (resp.\ RG3) if and only if there exits an associative product $x*y$ on $A$ and either an involutive automorphism or an involution $\sigma$ of $(A,*)$ such that $xy = x*\sigma(y)$ (resp.\  $xy = \sigma(x)*y$).
\end{prop}
\begin{proof}
Either of the claims follows from the other by taking the opposite algebra. Assume that
$(A,xy)$ is of type LG3 and denote by $e$ the right unit element. Linearizing $x(y(zy)) =
(x(yz))y$ at $y = e$ we get that $\lt_x \lt_e \lt_z = \lt_{x(ez)}$. It follows that $A$
with the product $x*y = x(ez)$ is an associative algebra with $e$ as its identity element,
and that the square of $\sigma = \lt_e$ is $\I_A$. It remains to prove that $\sigma$ is
either an automorphism or an involution of $(A,*)$. 

The identity $x(y(zy)) = (x(yz))y$ can be written as 
\begin{align*}
  \lt^*_x \sigma \lt^*_y \sigma \rt^*_{\sigma(y)} &= 
  \rt^*_{\sigma(y)} \lt^*_x \sigma \lt^*_y \sigma \,, \\
  \intertext{which is equivalent to}
  \sigma \lt^*_y \sigma \rt^*_{\sigma(y)} &= \rt^*_{\sigma(y)} \sigma \lt^*_y \sigma \,.
\end{align*}
Evaluating at $e$ we get that $\sigma(y*y) = \sigma(y)*\sigma(y)$. This proves the result in case that $\dim A = 1$ or $2$. So we will assume that $\dim A = 4$, \ie, $\H = (A,*)$. The equality $\sigma(y*y) = \sigma(y)*\sigma(y)$ also implies that $\sigma$ is an isometry of the multiplicative quadratic form of $\H$. Since isometries of $\H$ that fix the unit element are automorphisms or anti-automorphisms, the result follows.
\end{proof}

\begin{prop}
\label{prop:LC2}
  A  real division algebra $(A,xy)$ is of type LC2 (resp.\ RC2) if and only if there exists an associative product $x*y$ on $A$ such that either
  \begin{enumerate}
    \item $xy = x*\sigma(y)$ (resp.\ $xy = \sigma(x)*y$) for some involutive automorphism $\sigma$ of $(A,*)$ or
    \item $(A,x*y) \cong \C$ and $xy=\alpha(x)*\bar{y}$ (resp.\ $xy = \bar{x}*\alpha(y)$) for some bijective linear map $\alpha$ that fixes the unit element of $(A,*)$ ($x\mapsto \bar{x}$ denotes the complex conjugation).
  \end{enumerate}
\end{prop}
\begin{proof}
Substituting $x\backslash (x\backslash x)$ in $x(x(yz)) = (x(xy))z$ for $y$ we get that $A$ has inversion on the left. Thus, given an idempotent $e$, the product $x*y = (x/e)(e\backslash y)$ defines an alternative algebra, $\alpha = \rt_e$ fixes the unit element of $(A,*)$ and $\sigma = \lt_e$ is an involutive automorphism of $(A,*)$. The identity $x(x(yz)) = (x(xy))z$ is equivalent to
\begin{equation}
\label{eq:LC2aux1}
  x*(\sigma(x)*(\alpha(y)*z)) = \alpha(x*(\sigma(x)*y))*z.
\end{equation}
Setting $y = e$ in (\ref{eq:LC2aux1}) we get that $x*(\sigma(x)*z) = (x*\sigma(x))*z$
an with $z=e$ this means that $\alpha(x*\sigma(x)) = x * \sigma(x)$. 

If $\dim A = 8$ then the latter equality implies that $\sigma(x) \in \R x$ for all 
$x \in \Im(A,*)$. Thus $\sigma$ is either the identity on $A$ or the standard
involution. Since $\sigma$ is an automorphism we get that $\sigma = \I_A$. From
Equation~(\ref{eq:LC2aux1}) we get that $x*(x*(\alpha(y)*z)) = \alpha(x*(x*y))*z$. Using
the alternative law and evaluating at $z = e$ we get $x^{*2}*\alpha(y) = \alpha(x^{*2}*y)$
and $x^{*2}*(\alpha(y)*z) = (x^{*2}*\alpha(y))*z$. Since any element of $(A,*)$ can be
written as $x^{*2}$ for some adequate $x$, it follows that $(A,*)$ is associative, which
contradicts that $\dim A = 8$. 
Hence $\dim A\le 4$.

Now, as $(A,*)$ is associative, (\ref{eq:LC2aux1}) is equivalent to 
$x*\sigma(x)*\alpha(y) = \alpha(x*\sigma(x)*y)$. In case that $\dim A = 4$ then 
$\sigma(x) = a*x*a^{-1}$ for certain $a \in A$. The set $\{x*\sigma(x) \mid x \in A\}$
equals $\{x*a*x*a \mid x \in A\}*(a^{-1})^{*2} = A*(a^{-1})^{*2} = A$. 
Hence, $x*\alpha(y) = \alpha(x*y)$ for any $x,y \in A$. Substituting $e$ for $y$ in this
latter equation we get that $\alpha = \rt^*_{\alpha(e)} = \rt^*_e = \I_A$.

We are left with the case in which $\dim A = 2$. If $\sigma = \I_A$ then we can proceed as before to get that $\alpha = \I_A$. Otherwise $\sigma$ is the complex conjugation and the statement follows.
\end{proof}

\begin{prop}
  A  real division algebra $(A,xy)$ is of type LC4 or LG1 (resp.\ RC4 or RG1) if and only if there exist an associative product $x*y$ on $A$ and an involutive automorphism $\sigma$ of $(A,*)$ such that $xy = x*\sigma(y)$ (resp.\ $xy = \sigma(x)*y$).
\end{prop}
\begin{proof}
Since the variety LC4 is contained inside the variety LC2, the multiplication in $A$
can be written according to either (1) or (2) in Proposition~\ref{prop:LC2}. 
Existence of right unity in $A$ excludes the second possibility, and it is easy to check
that the product $xy=x*\sigma(y)$ of Proposition~\ref{prop:LC2}(1) satisfies the defining
identity $x(y(yz))=(x(yy))z$.
\end{proof}

\subsection{Balanced quasigroup identities} \label{sec:multilinear}

In many cases Corollary \ref{cor:PI} can be applied to  real division algebras satisfying
a balanced quasigroup identity to show that they are isotopes of Hurwitz
algebras. However, it turns out that a direct approach produces more precise results in
this case. Belousov's theorem about balanced identities of quasigroups also applies
here. However Belousov's paper does not completely cover all the cases that we are
interested in.

We shall write $p\stackrel{A}{=}q$ to indicate that $p=q$
holds in $A$, \ie, $p(a_1,\dots,a_n) = q(a_1,\dots,a_n)$ for all $a_1,\dots,a_n\in A$.
The \emph{support} of $p(x_1,\ldots,x_n)$ is the set $\supp(p)=\{x_1,\ldots,x_n\}$ of
variables occuring in $p$.

\begin{lma}
\label{lem:xy-yx}
Let $(A,xy)$ be a division algebra, $e$ a non-zero idempotent in $A$ and $x*y = (x/e)(e\backslash y)$. If there exist $\alpha, \beta \in \gl(A)$ such that $\alpha(e) = \beta(e) = e$ and $xy = \alpha(y)\beta(x)$ for any $x,y \in A$ then $(A,*)$ is commutative.
\end{lma}
\begin{proof}
With $x = e$ (resp.\ $y = e$) we obtain that $\alpha(y) = (ex)/e$ (resp.\ $\beta(x) = e \backslash (xe)$). Hence $(y/e) (e\backslash x) = (x/e)(e \backslash y)$, \ie, $(A,*)$ is commutative.
\end{proof}

\begin{prop}
\label{prop:M}
Let $(A,xy)$ be a division algebra with a non-zero idempotent $e$ and $x*y = (x/e)(e\backslash y)$. If $(A,xy)$ satisfies a balanced quasigroup identity then at least one of the following  statements holds for all $x,y,z \in A$:
\begin{itemize}
\item[(M0)] $(A,*)$ is commutative,
\item[(M1)] $(xy)z = \alpha_z(x) \beta(y)$,
\item[(M2)] $(xy)z = \alpha_z(y) \beta(x)$,
\item[(M3)] $(xy)z = \alpha(x) \beta_z(y)$ or
\item[(M4)] $(xy)z = \alpha(y) \beta_z(x)$
\end{itemize}
for some $\alpha,\beta \in \gl(A)$ and bilinear maps 
$A\times A\to \gl(A),\: (z,x) \mapsto \alpha_z(x), \beta_z(x)$ with 
$\alpha(e) = \beta(e) = \alpha_e(e) = \beta_e(e) = e$.
\end{prop}
\begin{proof}
Let $p \stackrel{A}{=} q$ be a balanced quasigroup identity satisfied by $A$, of
minimal degree $n$. Assume that $(A,*)$ is not commutative; this implies $n\geq3$.
Write $p = p_1p_2$ and $q = q_1 q_2$, where each $p_1,p_2,q_1,q_2$ each have degree at
least one.
Suppose that $\supp(p_1)=\supp(q_1)$. Then $\supp(p_2)=\supp(q_2)$ and, since $A$ is a
division algebra, evaluation at $e$ of the variables in $\supp(p_2)$ gives
$p_1\stackrel{A}{=}q_1$. Similarly, $p_2\stackrel{A}{=}q_2$. At least of the identities
$p_1=q_1$ and $p_2=q_2$ is non-trivial, which contradicts the minimality of the degree $n$
of $p$ and $q$.

Suppose instead that $\supp(p_1) \subseteq \supp(q_2)$.
Choose variables $x\in\supp(p_1)\subseteq\supp(q_2)$ and
$y\in\supp(q_1)\subseteq\supp(p_2)$. Now there exist maps $\alpha,\beta \in \gl(A)$ with
$\alpha(e) = e = \beta(e)$ (obtained by evaluating each $z\in\supp(p)\setminus\{x,y\}$ at
$e$ in $p(x_1,\ldots,x_n)=q(x_1,\ldots,x_n)$ and rewriting the
expression) such that $xy = \alpha(y) \beta(x)$ for any $x,y \in A$. Lemma \ref{lem:xy-yx}
implies that (M0) holds in this case. Therefore, after possibly interchanging the roles of
$p$ and $q$, we may assume that $\supp(p_1)\not \subseteq \supp(q_1)$ and
$\supp(p_1) \not\subseteq \supp(q_2)$. In particular the cardinality of $\supp(p_1)$ is
at least 2. Writing $p_1 = p_{11}p_{12}$, we have
\begin{equation}
\label{eq:many_factors}
  (p_{11}p_{12}) p_2 \stackrel{A}{=}  q_1 q_2 \,.
\end{equation}
Our assumption about $\supp(p_1)$ implies that there is one variable $x$ in
$\supp(p_{11})$ and another variable $y$ in $\supp(p_{12})$ such that one of them belongs
to $\supp(q_1)$ and the other belongs to $\supp(q_2)$. We select $x, y$ and a third
variable $z \in \supp(p_2)$. Evaluation of all variables different from $x, y$ and
$z$ at $e$ in (\ref{eq:many_factors}) now easily leads to one of the cases (M1), (M2),
(M3) and (M4), depending on the identity $p=q$.
\end{proof}

We will say that a division algebra $A$ with a non-zero idempotent $e$ is of \emph{type} M1, M2, M3 or M4 in case that $A$ satisfies the corresponding statement in the previous proposition.

\begin{rmk}
Observe that division algebras satisfying a quasigroup identity where at least three of
the variables appear with degree one might also be of one of the types M1, M2, M3 or
M4. For instance, the identity $(x(wy))(zw) = (xw)((yz)w)$ can be written as
$(xy)z = (xw)(((w\backslash y)(z/w))w)$. With $w = e$ we obtain that any division algebra
satisfying this identity and having a non-zero idempotent $e$ is of type M3.
\end{rmk}

\begin{prop}
\label{prop:mqi_classification}
Let $(A,xy)$ be a division algebra, $e$ a non-zero idempotent in $A$ and $x*y = (x/e)(e\backslash y)$:
\begin{enumerate}
\item If $A$ is of type M1 or M2 then $(A,*)$ is commutative.
\item If $A$ is of type M3 then $(A,*)$ is associative and $\rt_e$ is an automorphism of $(A,*)$.
\item If $A$ is of type M4 the $(A,*)$ is associative and $\rt_e$ is an anti-automorphism of $(A,*)$.
\end{enumerate}
\end{prop}
\begin{proof}
If $A$ is of type M1 then $(xy)z = \alpha_z(x) \beta(y)$. With $x=e$ we obtain that $yz =
\alpha_z(e)\beta(e\backslash y) = \alpha'(z)\beta'(y)$ for some $\alpha',\beta' \in
\gl(A)$ with $\alpha'(e) = e = \beta'(e)$. By Lemma \ref{lem:xy-yx} this implies
commutativity of $(A,*)$. If $A$ is of type M2 then $(xy)z = \alpha_z(y) \beta(x)$
which, with $y = e$,  implies that $xz = \alpha_z(e)\beta(x/e)$, which also leads to
$(A,*)$ being commutative.

Observe that
\[
 \Atp(A,*) = \{ (\varphi_1,\rt_e\varphi_2\rt^{-1}_{e}, \lt_e \varphi_3 \lt^{-1}_{e}) \mid (\varphi_1,\varphi_2,\varphi_3) \in \Atp(A)\}.
\]
If $A$ is of type M3 or M4 then $(\rt_z\rt^{-1}_e, \I_A,\beta_z\beta^{-1}_{e})\in\Atp(A)$
for all $z\in A$, implying that
$(\rt_z\rt^{-1}_e, \I_A,\lt_e\beta_z\beta^{-1}_{e}\lt_e\inv) \in \Atp(A,*)$. Since $e$ is
the unit element of $(A,*)$ and $\rt_z\rt^{-1}_e=\rt_{ez}^*$,
the condition of autotopy implies that
$\lt_e\beta_z\beta^{-1}_{e}\lt_e\inv =\rt_z\rt^{-1}_e$, so $ez $ belongs to the right
associative nucleus of $(A,*)$ for all $z$. Hence $(A,*)$ is associative.

Let us now prove the assertions about $\rt_e$. If $A$ is of type M3 then $(xy)z =
\alpha(x) \beta_z(y)$. With $x = e$ we obtain that $\beta_z(y) = e
\backslash((ey)z)$. With $y = z = e$ we get $\alpha(x) = xe$, so  $(xy)z = (xe)
(e\backslash((ey)z))$. In particular, $((x/e)(e\backslash y))z = x(e\backslash
(yz))$. Inserting $z=e$ into this identity gives $(x*y)e = (xe)*(ye)$, \ie, $\rt_e$ is
an automorphism of $(A,*)$. If $A$ is of type M4 then $(xy)z = \alpha(y)
\beta_z(x)$. With $y = e$ we obtain $\beta_z(x) = e \backslash ((xe)z)$. The same formula
with $x = z = e$ gives $\alpha(y) = ey$. Thus, $(xy) z = (ey)(e \backslash
((xe)z))$. With $z = e$ we finally obtain that
$\rt_e(x*y)=((x/e)(e\backslash y)) e = y (e\backslash(xe))=\rt_e(y)*\rt_e(x)$, that is,
$\rt_e$ is an anti-automorphism of $(A,*)$.
\end{proof}

\begin{rmk}
If $A$ is a division algebra with a non-zero idempotent $e$, satisfying a non-trivial
balanced quasigroup identity,
then the opposite algebra $A\op$ also satisfies a non-trivial balanced quasigroup
identity, and $e$ is an idempotent of $A\op$.
Hence Proposition~\ref{prop:mqi_classification} applied to the opposite algebra may
provide additional information about $\lt_e$.
\end{rmk}

Combining Proposition~\ref{prop:M} and \ref{prop:mqi_classification}, we obtain the
following result on real division algebras satisfying balanced quasigroup identities.

\begin{thm}
\label{cor:neat}
 Let $(A,xy)$ be a  real division algebra. If $(A,xy)$ satisfies a
 non-trivial balanced
  quasigroup identity then $(A,xy)$ is isotopic to $\R, \C$ or $\H$. In
 the latter case, the product $xy$ on $A$ can be obtained from the product $x*y$ of $\H$
 by $xy = \sigma(x)*\tau(y)$ with each of $\sigma$ and $\tau$ being either an automorphism
 or an anti-automorphism of $\H$.
\end{thm}

\begin{exa}
  Let $(A,xy)$ be a  real division algebra that satisfies the identity
\begin{equation} \label{example30}
  ((x_1 x_2)x_3)x_4 = x_3((x_2 x_1)x_4).
\end{equation}
We can briefly explore the structure of this algebra as follows. $(A,xy)$ is of type M4 so $(A,*)$ is associative and $\rt_e$ is an anti-automorphism of $(A,*)$ where $x*y = (x/e)(e\backslash y)$ and $e$ is a non-zero idempotent of $(A,xy)$. The opposite algebra $(A,xy)\op$ if of type M3, so $\lt_e$ is an automorphism of $(A,*)$. The product on $A$ is recovered as $xy = \sigma(\bar{x})*\tau(y)$ for certain automorphisms $\sigma,\tau$ of $(A,*)$, where $x\mapsto \bar{x}$ denotes the standard involution of the Hurwitz algebra $(A,*)$.

The identity satisfied by $(A,xy)$ can be written as
$\sigma\tau(\bar{x}_3)*\sigma^3(\bar{x}_1)*\sigma^2\tau(x_2)*\tau(x_4) =
\sigma(\bar{x}_3)*\tau\sigma \tau(\bar{x}_1)*\tau\sigma^2(x_2)*\tau^2(x_4)$. This is
equivalent to $\tau = \I_A$ and $\sigma^2 = \I_A$.
In the four-dimensional case, there exists a unique conjugacy class in $\Aut(\H)$ of
automorphisms of order two, represented by the map $\s'_0$ sending $\bi$ and $\bj$ to $-\bi$
and $-\bj$ respectively while fixing $1$ and $\mathbf{k}=\bi\bj$. Now $\s_0=\s'_0\kappa$ is
the reflection in the hyperplane $\spann\{1,\bi,\bj\}$.

Hence, to summarize, there exist precisely five isomorphism classes of real
 division algebras satisfying the identity (\ref{example30}),
represented by the following algebras:
$$ \R \,, \quad \C \,, \quad \C_{\kappa,\I_\C} \,, \quad \H \,, \quad \H_{\s,\I_\H} \,.$$
\end{exa}

\begin{exa}
We define the nonassociative words
\begin{eqnarray*}
 && p_1(x_1,x_2) = x_1 x_2 \\
&& p_n(x_1,\dots, x_{2^n}) = p_{n-1}(x_1,\dots, x_{2^{n-1}})p_{n-1}(x_{2^{n-1}+1},\dots, x_{2^{n}}) \quad (n\geq 2)
\end{eqnarray*}
and study division algebras that for some $n \geq 1$ satisfy the identity
\begin{equation}
\label{eq:auto_duplicated}
 yp_n(x_1,\dots,x_{2^n}) = p_n(y,x_1,\dots,x_{2^n-1})x_{2^n}
\end{equation}
and posses non-zero idempotents. Examples of the identities under consideration are
\begin{eqnarray*}
 (n=1) && y(x_1 x_2) = (yx_1)x_2,\\
(n = 2) && y((x_1x_2)(x_3x_4)) = ((yx_1)(x_2x_3))x_4,\\
(n=3) && y(((x_1x_2)(x_3x_4))((x_5x_6)(x_7x_8))) = (((yx_1)(x_2x_3))((x_4x_5)(x_6x_7)))x_8.
\end{eqnarray*}
Let $(A,xy)$ be such a division algebra. Notice that $A\op$ also satisfies the same
identity as $A$.
Since $A$ and $A\op$ are of type M3, the algebra $(A,*)$ is associative,
$\sigma = \rt_e$ and $\tau = \lt_e$ are automorphisms of $(A,*)$ and
$xy = \sigma(x)*\tau(y)$. Evaluating
(\ref{eq:auto_duplicated}) at $x_1 = x_2 =\cdots =x_{2^n}=e$ we obtain that $ye = \rt^{n+1}_e(y)$
so $\s^n=\rt^n_e = \I_A$, and by considering $A\op$ instead of $A$ we get $\tau^n=\lt^n_e=\I_A$.
With $y = e = x_1 = x_3 = \cdots = x_{2^n}$ the identity (\ref{eq:auto_duplicated}) gives
\begin{eqnarray*}
 e((ex_2)/e) &=& ep_n(e,x_2,e,\dots,e) = p_n(e,e,x_2,e,\dots,e)e\\
 &=& (e(x_2 e))/e
\end{eqnarray*}
so $\lt_e\rt^{-1}_e\lt_e = \rt^{-1}_e\lt_e\rt_e$, \ie, $\lt_e \rt_e \lt^{-1}_e = \rt_e \lt_e \rt^{-1}_e$. The subgroup generated by $\sigma$ and $\tau$ in the group of automorphisms of $(A,*)$ is  a quotient of the group
\[
 \langle g, h \mid g^n = 1, h^n = 1, ghg^{-1} = hgh^{-1} \rangle.
\]
With $u = h$ and $v = h^{-1}g$ we obtain the presentation
\[
 \langle u,v \mid u^n = 1, v^{2^n -1} = 1, u^{-1}v u = v^2 \rangle,
\]
a semidirect product $\cc_{2^n -1} \times_\varphi \cc_n$ of two cyclic groups $\cc_{2^n-1}
= \langle v \rangle$, $\cc_n = \langle u \rangle$  of orders $2^n-1$ and $n$ respectively
and $\varphi \colon \cc_n \to \Aut(\cc_{2^n -1})$ the homomorphism of groups determined by
$\varphi(u)\colon v \mapsto v^2$. In particular, the subgroup of $\Aut(A,*)$ generated by $\sigma$ and
$\tau$ is finite.

\begin{prop}
A real division algebra $(A,xy)$ satisfies the identity (\ref{eq:auto_duplicated}) for $n \geq 1$ if and only  there exists an associative product $x*y$ on $A$ such that  $xy = \sigma(x)*\tau(y)$  with $\sigma, \tau \in  \Aut(A,x*y)$ and either
\begin{enumerate}
\item $\sigma = \tau$ and $\sigma^n = \I_A$,
   or \label{balodd}
 \item $n$ is even, and $(A,xy)$ has involutive inversion. \label{baleven}
\end{enumerate}
\end{prop}
In particular, an algebra with involutive inversion satisfies (\ref{eq:auto_duplicated})
for any even $n$.

\begin{proof}
Assume that $(A,xy)$ is a real division algebra that satisfies (\ref{eq:auto_duplicated}) for some $n \geq 1$ and let $e$ be a nonzero idempotent. We know that $xy = \sigma(x)*\tau(y)$ with  $(A,x*y)$ an associative algebra,  $\sigma  = \rt_e$,  $\tau = \lt_e$, $\tau^n = \sigma^n = \I_A$ and $\sigma \tau \sigma^{-1} = \tau \sigma \tau^{-1}$.

If $n = 1$ or $\dim A = 1$ then $A$ is associative, and there is nothing to prove. 
If $\dim A = 2$ then  $(A,x*y) \cong \C$. In this case the
only nontrivial automorphism of $(A,x*y)$ has order two, so either $n$ is odd and $\sigma
= \tau = \I_A$ or $n$ is even and we have to include the possibility $x*y =
\bar{x}*\bar{y}$ where $x\mapsto \bar{x}$ denotes the complex conjugation. In both cases
the proposition holds.

It remains to consider the case $(A,x*y)\cong \H$. 
We know that $\s^n=\tau^n=\I_A$, so if $\s=\tau$ then \ref{balodd} holds.
So, we may assume that $\sigma \neq \tau$. By the Skolem-Noether
Theorem,  $\sigma\colon x \mapsto a*x*a^{-1}$ 
and $\tau \colon x \mapsto b*x*b^{-1}$ for certain $a,b \in A$ with
$a^{*n} = \pm e$, $b^{*n} = \pm e$ and $a*b*a^{-1} = \pm b*a*b^{-1}$. Possibly
replacing $a$ by $-a$ we may assume that $a*b*a^{-1} = b*a*b^{-1}$ also, $a^{*2n} = e$,
$b^{*2n} = e$. The elements $c = b$, $d = b^{-1}*a$ satisfy $c^{*2n} = e$, $d^{*2^{2n}-1}
= e$ and $c^{-1}*d*c = d^{*2}$. Lemma 10 in \cite{He53} shows that: i) $c*d = d*c$ or ii)
$d^{-1}*c = c*d$. 
In case i) $d = e = 1_A$, hence $a = b$ and $\sigma = \tau$, a contradiction.
In case ii) we get $d^{*2}=c\inv*d*c=d\inv$, implying $d^{*3} = e$ and thus
$c^{-1}*d*c=d^{-1}$. In particular $c^{*2}*d = d* c^{*2}$. We may assume that
$c*d \neq d*c$ (otherwise $d=e$ and $\sigma = \tau$), meaning that $e,c,d$ are linearly independent.
Therefore, $c^{*2}*d = d*c^{*2}$ implies $b^{*2}  \in \R e$, \ie, $\tau^2=\I_A$. 
Now, squaring each side of the equation $\s\tau\s\inv=\tau\s\tau\inv$ gives
  $\s^2=\I_A$, and hence $\s\tau\s=\tau\s\tau$,
\ie, $(A,xy)$ has involutive inversion. 
In case that $n$ is odd it follows that $\sigma = \tau = \I_A$, a contradiction.

Conversely, assume that $(A,x*y)$ is associative, $\sigma, \tau \in \Aut(A,x*y)$ satisfy
(1) or (2) in the statement and let us consider the product $xy = \sigma(x)*\tau(y)$. 
In terms of the product $*$, we have
$p_n(x_1,\dots,x_{2^n})=\a_1(x_1)*\cdots*\a_{2^n}(x_{2^n})$ where
$\a_m=\sigma^{1-b_{n-1}}\tau^{b_{n-1}}\cdots \sigma^{1-b_0}\tau^{b_0}$,
$m\in\{1,\ldots,2^n\}$ and $b_nb_{n-1}\cdots b_1b_0$ being the binary expansion of $m-1$,
\ie, $m-1=\sum_{i=0}^{n-1}b_i2^i$.
Equation~(\ref{eq:auto_duplicated}) is satisfied if and only if $\s^{n+1}=\s$,
$\tau^{n+1}=\tau$ and $\tau\a_m=\s\a_{m+1}$ for all $m\in\{1,\ldots,2^n-1\}$. 
With $m=\sum_{i=0}^{n-1}c_i2^i$, the last identity is equivalent to
\begin{equation}
\label{eq:case(2)}
 \tau \sigma^{1-b_{n-1}}\tau^{b_{n-1}} \cdots \sigma^{1-b_0}\tau^{b_0} = 
 \sigma \sigma^{1-c_{n-1}}\tau^{c_{n-1}}   \cdots \sigma^{1-c_0}\tau^{c_0} .
\end{equation}

For $\sigma = \tau$, Equation~(\ref{eq:case(2)}) reduces to $\s^{n+1}=\s^{n+1}$
for all $m$, so (\ref{eq:auto_duplicated}) is satisfied if and only if $\s^n=\I_A$.
Thus, it remains to consider the case with $n$ even and $(A,xy)$ having involutive
inversion, and we may assume that $\s\ne\tau$.
By Proposition~\ref{prop:identity}, the identity (\ref{eq:auto_duplicated}) holds for
$n=2$, so assume $n>2$. 

As $\s^2=\tau^2=\I_A$ and $n$ is even, $\s^n=\tau^n=\I_A$. We need to show that
(\ref{eq:case(2)}) holds for all $m\in\{1,\ldots,2^{n-1}-1\}$.
If $m$ is odd then $b_0 = 0$, $c_0 = 1$, while $b_i = c_i$ for all $i\ge1$.
The automorphism $\sigma^{1-b_{n-1}}\tau^{b_{n-1}} \cdots \sigma^{1-b_1}\tau^{b_1}$ involves
an odd number of occurrences of $\tau,\sigma$ so it must be either $\sigma, \tau$ or
$\sigma \tau \sigma$. Checking all these three cases we get that (\ref{eq:case(2)}) holds
if $m$ is odd. Let us assume that $m$ is even, so $b_0 = 1$ and $c_0 =0$. If $b_1 = 0$
then $c_1 = 1$ and $b_2 = c_2,\dots, b_{n-1} = c_{n-1}$. The automorphism
$\sigma^{1-b_{n-1}}\tau^{b_{n-1}} \cdots \sigma^{1-b_2}\tau^{b_2}$ involves an even number
of occurrences of $\sigma, \tau$ so it must be either $\I_A, \sigma \tau$ or
$\tau\sigma$. Again we can check all these three possibilities to obtain that
(\ref{eq:case(2)}) holds. 
We are left with the case of $b_0 = b_1 = 1$. Here, $c_0 = c_1 = 0$ and (\ref{eq:case(2)}) becomes
\begin{displaymath}
\tau \sigma^{1-b_{n-1}}\tau^{b_{n-1}}\cdots \sigma^{1-b_2}\tau^{b_2} = \sigma \sigma^{1-c_{n-1}}\tau^{c_{n-1}} \cdots \sigma^{1-c_2}\tau^{c_2} 
\end{displaymath}
which holds by induction on $n$.
\end{proof}
\end{exa}

\section{Proof of Theorem \ref{thm:PI}} \label{sec:proof}
Recall the statement of Theorem \ref{thm:PI}:
\begin{quotation}
\emph{Let $(A,*)$ be a  unital real division algebra. If there exist $h \in \gl(A)$ and $\iota \in \{1,2,3\}$ such that $\{ h \lt^*_x h^{-1} \mid x \in A \}$ or $\{ h \rt^*_x h^{-1} \mid x \in A\}$ is contained in $\Tder(A,*)_\iota$ then $(A,*)$ is a Hurwitz algebra.}
\end{quotation}
Using $A\op$ instead of $A$ it is clear that we only have to consider the case $\{ h \lt^*_x h^{-1} \mid x \in A \} \subseteq \Tder(A,*)_\iota$. This condition implies that $A$ is irreducible under the action of $\Tder(A,*)_\iota$.

\medskip

\emph{Since we will mainly deal with the product $x*y$, to avoid annoying notation, we
  will simply write $xy$ instead of $x*y$, while $e$ remains the unit element.}

\medskip

For any  real division algebra $A$ the Lie algebra $\Tder(A)$ decomposes as $\Tder(A) = \cS \oplus \cZ$ with $\cS$ a compact Lie algebra and $\cZ$ the center \cite{JP08}.

As before, we use the notation $\Tder(A)_i = \pi_i(\Tder(A))$, similarly
$\cS_i = \pi_i(\cS)$ and $\cZ_i = \pi_i(\cZ)$. We write
$\rC(X) = \{ \gamma \in \End_\R(A) \mid [\gamma, d] = 0 \: \forall_{d \in X} \}$ for the
centralizer in $\End_\R(A)$ of any subset $X \subseteq \End_\R(A)$.

\begin{lma}[\hbox{\cite[Proposition 1 and Corollary 2]{JP08}}]
 For any  real division algebra $A$ and any $i \in \{1,2,3\}$ we have that
 \begin{enumerate}
 \item $\Tder(A)_i = \cS_i \oplus \cZ_i$,
 \item $\cS_i$ is semisimple,
 \item $\cZ_i$ is the center of $\Tder(A)_i$ and
 \item $\cZ_i$ consists of semisimple maps.
 \end{enumerate}
\end{lma}

\begin{lma}
\label{lem:centralizer}
 Let $A$ be a  real division algebra. If $\dim A \geq 4$ then $\dim \{ \lt_a \mid [\lt_a,\lt_x] = 0 \: \forall_{x \in A}\} \leq 1$.
\end{lma}
\begin{proof}
Let us assume that $\dim (\lt_A \cap \rC(\lt_A)) \geq 2$. Given $a\in A\setminus\{0\}$ with $\lt_a \in \lt_A \cap \rC(\lt_A)$ we define $x*y = (x/a)(a\backslash y)$. The new algebra $(A,*)$ is a unital division algebra with left multiplication operators $\lt^*_x = \lt_{x/a}\lt_a^{-1}$ and $\dim \{ \lt^*_x \mid [\lt^*_x,\lt^*_y] = 0 \:\forall_{y \in A}\} \geq 2$. Therefore, without loss of generality we may assume that $A$ is unital.

Let $\lt_a \in \lt_A \cap \rC(\lt_A)$ with $\lt_a \not \in \R \I_A$ and
$C = \R \I_A\oplus \R \lt_a$. Since by Schur's Lemma $\rC(\lt_A)$ is a division algebra,
$\rC(\lt_A) \cong \C$ or $\H$. The set $C$ is in fact a subalgebra, isomorphic to the
complex numbers, of the quadratic algebra $\rC(\lt_A)$, $A$ is a left $C$-vector space and
the maps in $\lt_A$ are $C$-linear.

Given $x \in A$, let $\alpha \I_A + \beta \lt_a$ be an eigenvalue of $\lt_x$ in $C$ and $v$ a corresponding non-zero eigenvector. We have that $xv = (\alpha \I_A + \beta \lt_a)v = (\alpha e + \beta a) v$. Since $A$ is a division algebra  we can conclude that $x \in \R e \oplus \R a$ so $\dim A \leq 2$, a contradiction.
\end{proof}

\medskip

\emph{In what follows, $A$ is an algebra that satisfies the hypotheses of Theorem \ref{thm:PI} as established at the beginning of this section. The value of $\iota$ and the linear map $h$ are fixed.}

\medskip

For any vector space $X$ over $\R$, ${_\C}X$ will denote the complex vector space
$\C \otimes_{\R}X$. Recall \cite{BO81b} that in any dimension $\equiv 0,1, 2 \mod 4$ there
exists precisely one isomorphism class of irreducible $\su(2)$-modules.
That of dimension $2m +1$, $W(2m)$, is absolutely irreducible and
${_\C}W(2m) \cong V(2m)$. The one $V(2n-1)$ in dimension $4n$ -- the complex
$\Lsl(2,\C)$-module $V(2n-1)$ seen as a real $\su(2)$-module -- satisfies
${_\C}V(2n-1) \cong V(2n-1) \oplus V(2n-1)$, so $\End_{\su(2)}(V(2n-1))\cong \H$.

\begin{lma}
\label{lem:isotypic}
For any non-zero ideal $I \trianglelefteq \Tder(A)_\iota$ the $I$-module $A$ has only one isotypic component.
\end{lma}
\begin{proof}
 Since $\Tder(A)_\iota$ is reductive, any isotypic component of $A$ is stable under the action of $\Tder(A)_\iota$, so it must be the whole $A$.
\end{proof}

\begin{lma}
\label{lem:dimension}
If $\dim A \geq 4$ then $\dim \cS_\iota \geq \dim A -1$. In particular, $\Tder(A)$ is not abelian.
\end{lma}
\begin{proof}
Let $\pi$ denote the restriction of the projection $\cS_\iota \oplus \cZ_\iota \to
\cS_\iota$ to $h \lt_A h^{-1}$. The elements in the kernel of $\pi$ are of the form $h
\lt_a h^{-1}$ with $[h \lt_a h^{-1}, h \lt_x h^{-1}] = 0$ for any $x \in A$. By Lemma
\ref{lem:centralizer} we obtain that $\dim(\ker \pi) \leq 1$ and
$\dim \cS_\iota \geq \dim (h \lt_A h^{-1}) - \dim (\ker \pi) \geq \dim A - 1$.
\end{proof}

\begin{lma}
 We have that
\[
\rC(\Tder(A)_\iota) \subseteq \{ h \rt_a h^{-1} \mid a \in \N_r(A) \}.
\]
\end{lma}
\begin{proof}
 Clearly $\rC(\Tder(A)_\iota) \subseteq  h \{ \gamma \in \End_\R(A) \mid [\gamma,\lt_x] = 0 \: \forall_{x\in A}\} h^{-1}$. The existence of unit element in $A$ easily implies that any $\gamma$ that commutes with all the maps in $\lt_A$ must be of the form $\rt_a$ with $a \in \N_r(A)$.
\end{proof}

In case that $\dim A = 1,2$ the existence of unit element implies that $A$ is a Hurwitz
algebra, so we may assume that $\dim A \geq 4$.

\begin{prop}
If $\dim A = 4$ then $A$ is isomorphic to the quaternions.
\end{prop}
\begin{proof}
By \cite[comment after Theorem 10]{JP08} either $A$ is isotopic to $\H$, and therefore it
is isomorphic to $\H$ \cite[Theorem~12]{albert42a}, or $\cS \cong \su(2)$ and
$\dim \Tder(A)  = \dim \cS + 2$. In the latter case $\dim \Tder(A)_\iota = 4$ so
$\Tder(A)_\iota= h \lt_A h^{-1}$. In particular, $\lt_A$ is a Lie algebra isomorphic to
$\su(2)$. Since the action of $\lt_A$ on $A$ is faithful, $A \cong V(1)$, the unique
(up to isomorphism) irreducible $\su(2)$-module of dimension four.  This proves that
$\{ \rt_a \mid a \in \N_r(A)\} = \End_{\lt_A}(A) \cong \End_{\su(2)}(V(1)) \cong\H$.
Therefore, $A = \N_r(A) \cong \H$.
\end{proof}

In the following we will assume that $\dim A = 8$. If $A$ is not isomorphic to the
octonions $\O$ then, by \cite[p. 2206]{JP08}, the simple ideals that can occur in the
decomposition of $\cS$ as a direct sum of simple ideals are Lie algebras of the types
\[
 G_2,\quad \su(3),\quad B_2, \quad \text{and} \quad \su(2).
\]
Observe that the kernels of the projections $\pi_1, \pi_2$ and $\pi_3$ have dimensions
$\leq 4$, since they are isomorphic to $\N_m(A), \N_r(A)$ and $\N_l(A)$ respectively. All
restrictions of these projections to ideals of type $G_2, \su(3)$ or $B_2$ are therefore
injective.

\begin{lma}
The Lie algebra $\Tder(A)$ does not contain ideals isomorphic to compact Lie algebras of type $G_2$ or $\su(3)$.
\end{lma}
\begin{proof}
Assume, on the contrary, that $\Tder(A)$ contains an ideal isomorphic to a Lie algebra of type $G_2$. The projection of that ideal is an ideal $I_\iota$ of $\cS_\iota$  of type $G_2$. Thus $A$ is an $I_\iota$-module isomorphic to the direct sum of two non-isomorphic modules (a trivial one-dimensional module and an absolutely irreducible seven-dimensional one), which is impossible by Lemma \ref{lem:isotypic}.

If  $\Tder(A)$ contains an ideal isomorphic to a Lie algebra of type $\su(3)$ then
$\cS_\iota$ contains an ideal $I_\iota$ isomorphic to $\su(3)$. Lemma \ref{lem:isotypic}
implies that $A$ is an absolutely irreducible eight-dimensional module of $I_\iota$. Thus,
$\Tder(A)_\iota = I_\iota \oplus \R \I_A$ and $\dim h \lt_A h^{-1} \cap I_\iota= 7$.
However, after extending scalars ${_\C}A$ is isomorphic to the adjoint module of
${_\C}I_\iota \cong \Lsl(3,\C)$, so the kernel of any map in $I_\iota$ should be non-zero,
which is not the case for non-zero maps in $h \lt_A h^{-1}$, a contradiction.
\end{proof}

\begin{lma}
If 
$\Tder(A)$ does not contain an ideal isomorphic to a compact Lie
algebra of type $B_2$, then $A$ is isomorphic to the octonions.
\end{lma}
\begin{proof}
The hypothesis implies that  $\cS_i$ decomposes as a direct sum of ideals isomorphic to
$\su(2)$. Since $\dim \cS_i \geq 7$ and $\cS_i$ cannot contain a direct sum of four ideals
isomorphic so $\su(2)$ \cite[p. 2205]{JP08}, it follows that
$\cS_\iota = I^{(1} \oplus I^{(2} \oplus I^{(3}$ with $I^{(i}\cong \su(2)$, $i = 1,2,3$.
By Lemma \ref{lem:isotypic}, $A$ must be either irreducible or a direct sum of two
irreducible $I^{(1}$-modules isomorphic to $V(1)$.
In the first case, $I^{(2}\oplus I^{(3} \subseteq \rC(I^{(1}) \cong \R, \C, \H$
gives a contradiction. Therefore, $A \cong V(1) \oplus V(1)$ as an $I^{(i}$-module ($i
=1,2,3$). This proves that  ${_\C}A \cong V(1) \otimes V(1) \otimes V(1)$ as a
${_\C}\cS_\iota$-module. In particular  $A$ is an absolutely irreducible
$\cS_\iota$-module, $\cZ_\iota = \R \I_A$ and $\dim \cS_\iota \cap h \lt_A h^{-1} = 7$. By
dimension counting $h \lt_A h^{-1} \cap I^{(i}\neq 0$ for  $i=1,2,3$.

Choose non-zero elements $a_1,a_2\in A$ such that
$h \lt_{a_i}h^{-1} \in h \lt_A h^{-1} \cap I^{(i}$.
Since  $I^{(1} \subseteq \rC(I^{(2} \oplus I^{(3}) \cong \H$, $\lt_{a_1}$ generates a
subalgebra $C = \R \I_A \oplus \R \lt_{a_1}$ that is isomorphic to $\C$, and $\lt_{a_2}$
is $C$-linear. Let $\alpha \I_A + \beta \lt_{a_1}\in C$ be an eigenvalue of $\lt_{a_2}$,
and $v \in A$ a corresponding eigenvector. Now $a_2 v = (\alpha e + \beta a_1)v$,
 but this implies that $a_2 = \alpha e + \beta a_1$, which is not possible.
\end{proof}

It only remains to eliminate the possibility that $\Tder(A)$ contains ideals isomorphic
to a compact Lie algebra of type $B_2$. Assume on the contrary that  $\cS_\iota$ contains
such an ideal $I_\iota$. Then $A$ is either irreducible as an $I_\iota$-module (but not
absolutely irreducible) or isomorphic to a direct sum of a three-dimensional trivial
$I_\iota$-module and a five-dimensional absolutely irreducible one \cite[Subsection
4.1]{JP08}. Lemma \ref{lem:isotypic} rules out the latter possibility.  Let $I$ be the
ideal of $\cS$ that projects isomorphically onto $I_\iota$ by $\pi_\iota$. Two of the
$I$-modules $A_1$, $A_2$, $A_3$ are irreducible and the other decomposes as the direct sum
of a three-dimensional trivial $I$-module, which we will denote by $T$, and a
five-dimensional absolutely irreducible one, which we will denote by $T^\perp$
\cite[Subsection 4.1]{JP08}. After extending scalars to $\C$, the two eight-dimensional
irreducible modules are isomorphic to $V(\lambda_2)$, and $_\C{T}^\perp$ is isomorphic to
$V(\lambda_1)$, where $\lambda_1, \lambda_2$ are the fundamental weights relative to some
Cartan subalgebra of $_\C{I}$ \cite{Hu78}. 
Also observe that the kernels of the
projections $\pi_1, \pi_2, \pi_3$ are ideals of dimension $\leq 4$, so they commute with
$I$.

\begin{lma}
\label{lem:nuclei}
If $\Tder(A)$ contains an ideal $I$ isomorphic to a compact Lie algebra of type $B_2$ then
$\N_l(A) = \N_m(A) = \N_r(A) = \R e$ and the restrictions of the projections $\pi_1,\pi_2$
and $\pi_3$ to $\cS$ are injective.
\end{lma}
\begin{proof}
Any $(d'_1, d'_2, 0) \in \Tder(A)$ commutes with $I$ and it is of the form $(L_w,L_w,0)$
with $w \in \N_l(A)$. If $\lt_w \not\in \R \I_A$ then the $I$-modules $A_1$ and
$A_2$ must be irreducible (otherwise $\lt_w$ would act as a scalar multiple of the
identity on the absolutely irreducible $I$-module $T^\perp$, meaning that $w \in \R e$, a
contradiction) and $A_3 = T \oplus T^\perp$. Given $(d_1,d_2,d_3) \in I$ and $a \in T$,
$d_1(xa) = d_2(x)a + xd_3(a) = d_2(x)a$ implies that $\rt_b\rt_a^{-1}$ commutes with
$d_1$.  The dimension of $\spann_{\R}\{ \rt_b\rt_a^{-1} \mid a,b \in T\}$ is at least
$3$, so the associative subalgebra of $\End_\R(A)$ generated by this set is
$\rC(I_1)\cong\H$. The map $\lt_w$ belongs to $\rC(I_1)$ and it commutes with
$\rt_b\rt_a^{-1}$ for any $a,b$. Hence $\lt_w \in \R \I_A$, contradicting our
assumption. This proves that $\N_l(A)= \R e$ and $\ker \pi_3 = \R (\I_A,\I_A,0)$.
A similar argument proves that $\N_r(A) = \R e$ and
$\ker \pi_2 = \R (\I_A,0,\I_A)$. Since $\dim \rC(\Tder(A)_\iota) \leq \dim \N_r(A) = 1$ we
have $\cZ_\iota = \R\I$.

Any $(0,d'_2,d'_3) \in \Tder(A)$ is of the form $(0,\rt_w,-\lt_w)$ with
$w \in\N_m(A)$. In case that $\lt_w \not\in \R \I_A$, as $I$-modules, $A_2$ and $A_3$ must
be irreducible and $A_1 = T \oplus T^\perp$. This forces $\iota = 2$ or $\iota = 3$. 
If $\dim \N_m(A) = 2$ then $\{ (0, \rt_a, -\lt_a) \mid a \in \N_m(A)\}$ would be an
abelian ideal of dimension $\geq 2$ and thus $\dim \cZ_\iota \geq 2$, which is false.

If we instead assume that $\N_m(A) \cong \H$, then dimension counting gives
$\rC(I_2) = \rt_{\N_m(A)}$ and $\rC(I_3) = \lt_{\N_m(A)}$.
We now have enough information to determine the decomposition of $\Tder(A)$ as a direct
sum of ideals. Given $(d'_1,d'_2,d'_3) \in \Tder(A)$  with $[(d'_1,d'_2,d'_3),I]=(0,0,0)$, its image
$d'_3$ under $\pi_3$ belongs to $\rC(I_3)$ so it is of the form $-\lt_w$ for some $w \in \N_m(A)$.
Hence $(d'_1,d'_2 - \rt_w, 0) \in \Tder(A)$. If we denote by $T_m$ the set
$\{(0,\rt_w,-\lt_w) \mid w \in \N_m(A), \tr(\rt_w) = 0 = \tr(\lt_w) \}$, which is a Lie
ideal of $\Tder(A)$ isomorphic to $\su(2)$, then
\begin{displaymath}
\Tder(A) = I \oplus T_m \oplus \spann_{\R}\{ (\I_A,\I_A,0),(\I_A,0,\I_A)\}.
\end{displaymath}
The structure of the $_\C{I} \oplus {_\C{T}_m}$-modules $_\C{A}_1$, $_\C{A}_2$ and
$_\C{A}_3$ is given by
\begin{eqnarray*}
_\C{A}_1 &\cong & 3V(0) \otimes V(0) \oplus V(\lambda_1) \otimes V(0),\\
_\C{A}_2 &\cong & V(\lambda_2) \otimes V(1) \quad \text{and}\\
_\C{A}_3 &\cong & V(\lambda_2) \otimes V(1)
\end{eqnarray*}
where the first and second slots of the tensor products correspond to modules of $_\C{I}$
and $_\C{T}_m$ respectively. Observe that $A_1$ is a trivial $T_m$-module while each of
$A_2$ and $A_3$ decomposes as the direct sum of two irreducible (but not absolutely
irreducible) four-dimensional modules.

The decompositions above have been derived under the assumption that $\N_m(A)$ is isomorphic to $\H$.
However, will shall see that these decompositions cannot occur. Since
\begin{displaymath}
(V(\lambda_2) \otimes V(1)) \otimes (V(\lambda_2) \otimes V(1)) \cong (V(2 \lambda_2) \oplus V(\lambda_1) \oplus V(0)) \otimes (V(2) \oplus V(0))
\end{displaymath}
has only a copy of the trivial module $V(0)\otimes V(0)$ and the product on $_\C{A}$ is a
homomorphism $_\C{A}_2 \otimes _\C{A}_3 \to {_\C{A}_1}$ of $_\C{\Tder(A)}$-modules, the
product cannot be surjective ($_\C{A}_1$ has three copies of $V(0) \otimes V(0)$). This
contradicts the fact that $AA = A$. Therefore, $\N_m(A) = \R e$, as desired.
\end{proof}

\begin{lma}
\label{lem:types}
If $\Tder(A)$ contains an ideal $I$ isomorphic to a compact Lie algebra of type $B_2$ then there exists another ideal $J$ isomorphic to $\su(2)$ such that
\begin{displaymath}
\Tder(A) = I \oplus J \oplus \spann_{\R}\{ (\I_A,\I_A,0),(\I_A, 0, \I_A)\}.
\end{displaymath}
Moreover, two of the $_\C{I} \oplus {_\C{J}}$-modules $_\C{A}_1$, $_\C{A}_2$ and
$_\C{A}_3$ are isomorphic to $V(\lambda_2) \otimes V(1)$ while the other is isomorphic to
$V(\lambda_1) \otimes V(0) \,\oplus\, V(0) \otimes V(2)$.
The dimension of $\Hom_{I\oplus J}(A_2 \otimes A_3, A_1)$ is $2$.
\end{lma}
\begin{proof}
If $\Tder(A)_\iota = I_\iota \oplus \cZ_\iota$ then
$\dim \cZ_\iota \leq \dim \rC(\Tder(A)_\iota) \leq \dim \N_r(A) = 1$ so
$\Tder(A)_\iota = I_\iota \oplus \R \I_A$ and $\rC(\Tder(A)_\iota) = \rC(I_\iota)\cong\H$,
a contradiction. Thus there exists another semisimple ideal $J$ such that
\[
\Tder(A) = I \oplus J \oplus \spann_{\R}\{ (\I_A,\I_A,0),(\I_A, 0, \I_A)\}.
 \]
Since the projection $\pi_\iota$ is injective and $\pi_\iota(J) \subseteq \rC(I_\iota)$,
we have $J \cong \su(2)$. The extensions of the two irreducible $I$-modules in
$\{A_1,A_2, A_3\}$ have to be $_\C{I} \oplus {_\C{J}}$-modules isomorphic to
$V(\lambda_2)\otimes V(1)$. The other $I$-module is isomorphic to a direct sum of a
trivial three-dimensional $I$-module $T$ and an irreducible five-dimensional one
$T^\perp$. Since $J$ centralizes the action of $I$, both modules $T$ and $T^\perp$ are
stable under the action of $J$ and at least one of them has to be non-trivial, as the
projections of $J$ onto its components do not vanish. The only possibility is
${_\C{T}}\cong V(0) \otimes V(2)$ and ${_\C{T}^\perp} \cong V(\lambda_1) \otimes V(0)$.

Now the statement $\dim(\Hom_{I\oplus J}(A_2 \otimes A_3, A_1))=2$ is a consequence of the
following formulae for the decomposition of tensor products
\begin{eqnarray*}
V(\lambda_2) \otimes V(\lambda_2) &\cong & V(2\lambda_2) \oplus V(\lambda_1) \oplus V(0)\\
V(\lambda_1) \otimes V(\lambda_2) &\cong & V(\lambda_1 + \lambda_2) \oplus V(\lambda_2)\\
V(\lambda_2) \otimes V(0) & \cong& V(\lambda_2)
\end{eqnarray*}
and
\begin{eqnarray*}
V(1) \otimes V(1) & \cong& V(2) \oplus V(0)\\
V(2) \otimes V(1) &\cong& V(3) \oplus V(1)\\
\quad V(1) \otimes V(0) &\cong& V(1) \,.
\end{eqnarray*}
\end{proof}

Now that we understand the structure of $\Tder(A)$ and the $\Tder(A)$-modules $A_1$, $A_2$
and $A_3$, we may construct some models for them based on the octonions $\O$ to conclude
that $A$ has to be an isotope of $\O$. Since unital isotopes of the octonions are
isomorphic to the octonions then Theorem \ref{thm:PI} will follow.

Recall that $\Tder(\O)$ is isomorphic to the direct sum of a compact Lie algebra of type
$D_4$ and a two-dimensional center $\spann_{\R}\{ (\I_\O, \I_\O,0),(\I_\O,0,\I_\O)\}$.
The Principle of Local Triality (see \eg{} \cite[Section~3.5]{sv00}) implies that for any
$i\in \{1,2,3\}$ and any map $d_j$ that is skew-symmetric relative to
the bilinear form associated to the multiplicative quadratic form of $\O$, there exist
unique skew-symmetric
maps $d_j, d_k$ such that $(d_1,d_2,d_3) \in \Tder(\O)$ and $\{i,j,k\} = \{1,2,3\}$. Let
us denote the product of $\O$ by $x*y$ and consider a quaternion subalgebra 
$\H = \spann_{\R}\{ e, \bi, \bj , \bi*\bj\}$ with 
$\bi^{*2} = \bj^{*2} = (\bi*\bj)^{*2} = -e$. 
Fix $T=\spann_{\R}\{ e, \bi ,\bj \}$ and let $T^\perp$ be the orthogonal complement of $T$ in
$\O$.

By Lemma \ref{lem:types}, we have the following three cases:
\begin{itemize}
\item[Case 1.] $_\C{A}_1 \cong V(\lambda_1)\otimes V(0) \oplus V(0)\otimes V(2)$:
The subalgebras
\begin{eqnarray*}
\mathfrak{B}_2 &=& \{(d_1,d_2,d_3) \in \Tder(\O) \mid \tr(d_1) =  \tr(d_2)= 0 \textrm{ and } d_1\vert_{T} = 0\}\\
\mathfrak{su}_2 &=& \spann_{\R}\{ (\lt^*_\bi + \rt^*_\bi,\lt^*_\bi,\rt^*_\bi), (\lt^*_\bj+\rt^*_\bj,\lt^*_\bj,\rt^*_\bj), \\
&& \quad\quad ([\lt^*_\bi+\rt^*_\bi,\lt^*_\bj+\rt^*_\bj],[\lt^*_\bi,\lt^*_\bj],[\rt^*_\bi,\rt^*_\bj])\}
\end{eqnarray*}
are simple Lie algebras, $\mathfrak{B}_2$ is compact of type $B_2$, $\mathfrak{su}_2$ is
isomorphic to $\su(2)$, and they commute each other.
The projections $\Tder(\O) \to \End_\R(\O)$ $(d_1,d_2,d_3)\mapsto d_i$ ($i = 1,2,3$) provide
three representations, $\O_1, \O_2$ and $\O_3$, of $\Tder(\O)$. 
As ${_\C}\mathfrak{B}_2 \oplus {{_\C}\mathfrak{su}}_2$-modules,
\begin{eqnarray*}
{_\C}\O_1 &=& {{_\C}T} \oplus {{_\C}T^\perp} \cong V(0)\otimes V(1) \oplus V(\lambda_1) \otimes V(0),\\
{_\C}\O_2 &\cong& {_\C}\O_3 \cong  V(\lambda_2) \otimes V(1) \,.
\end{eqnarray*}
We may identify $A$ with $\O$, $I$ with $\mathfrak{B}_2$ and $J$ with $\mathfrak{su}_2$. With this identification the products  $x*y$ and $xy$ are homomorphisms $\O_2 \otimes \O_3 \to \O_1$ of $\mathfrak{B}_2 \oplus \mathfrak{su}_2$-modules. The maps
\begin{eqnarray*}
   \O_2 \otimes \O_3 &\to& \O_1 = T \oplus T^\perp\\
x \otimes y &\mapsto& \alpha \pi_T(x*y) + \beta \pi_{T^\perp}(x*y)
\end{eqnarray*}
with $\alpha,\beta \in \R$ and $\pi_T, \pi_{T^\perp}$ the projections onto $T$ and
$T^\perp$ parallel to $T^\perp$ and $T$ respectively, give all such homomorphisms, so
$xy = \alpha \pi_T(x*y) + \beta \pi_{T^\perp}(x*y)$ for some $\alpha,\beta \in \R$.
This implies that $A$ is an isotope of the octonions.
\item[Case 2.] $_\C{A}_2 \cong V(\lambda_1)\otimes V(0) \oplus V(0)\otimes V(2)$: The subalgebras that we consider in this case are
\begin{eqnarray*}
\mathfrak{B}_2 &=& \{(d,d',d) \in \Tder(\O) \mid \tr(d) =  0 \textrm{ and } 
d'\vert_{T} = 0\} \,,\\
\mathfrak{su}_2 &=& \spann_{\R}\{ (\lt^*_\bi,\lt^*_\bi+\rt^*_\bi,-\lt^*_\bi), (\lt^*_\bj,\lt^*_\bj+\rt^*_\bj,-\lt^*_\bj), \\
&& \quad\quad ([\lt^*_\bi,\lt^*_\bj],[\lt^*_\bi +
\rt^*_\bi,\lt^*_\bj+\rt^*_\bj],[\lt^*_\bi,\lt^*_\bj])\} \,.
\end{eqnarray*}
An argument similar to that in Case 1 shows that, after identifications, $xy = (\alpha \pi_T(x) + \beta \pi_{T^\perp}(x))*y$. Therefore, $A$ is an isotope of the octonions.
\item[Case 3.] $_\C{A}_3 \cong V(\lambda_1)\otimes V(0) \oplus V(0)\otimes V(2)$: The subalgebras that we consider in this case are
\begin{eqnarray*}
\mathfrak{B}_2 &=& \{(d,d,d') \in \Tder(\O) \mid \tr(d) =  0 \textrm{ and } 
d'\vert_{T} = 0\} \,,\\
\mathfrak{su}_2 &=& \spann_{\R}\{ (\rt^*_\bi,-\rt^*_\bi,\lt^*_\bi+\rt^*_\bi), (\rt^*_\bj,-\rt^*_\bj,\lt^*_\bj+\rt^*_\bj), \\
&& \quad\quad ([\rt^*_\bi,\rt^*_\bj],[\rt^*_\bi,\rt^*_\bj],[\lt^*_\bi +
\rt^*_\bi,\lt^*_j+\rt^*_\bj])\} \,.
\end{eqnarray*}
and the arguments are similar to those in Case 2.
\end{itemize}

\appendix
\section{Normal forms} \label{normalforms}

Here we list normal forms for the $\so(\R^3)$-sets $\mathbb{P}(\R^3)\times\R^3$,
$\R^3\times\R^3$, $\mathbb{P}(\R^3)\times\mathbb{P}(\R^3)\times\R^3$ and
$\mathbb{P}(\R^3)\times\R^3\times\R^3$, with respective group actions, defined by
(\ref{so3act1}) and (\ref{so3act2}), coming from $\mathscr{B}_{00}$, $\mathscr{B}_{01}$,
$\mathscr{B}_{10}$ and $\mathscr{B}_{11}$.
Let $N$ and $P$ be the sets of non-negative and positive real numbers, respectively, and
set $\mathcal{P}_1=\smatr{1\\\R\\\R}\cup\smatr{0\\1\\\R}\cup\smatr{0\\0\\1}$,
$\mathcal{P}_2=\smatr{P\\P\\\R}\cup\smatr{P\\0\\N}\cup\smatr{0\\N\\N}$.

\subsection{The case $(c,b,D_d,\b)\in\mathscr{B}_{00}$}

For $d\in\hat{\mathcal{T}}_1$:
\begin{align*}
  \mathcal{N}_{00}^1={}&\smatr{1&N\\0&N\\0&0}\subseteq\mathbb{P}(\R^3)\times\R^3\,.
\end{align*}
For $d\in\hat{\mathcal{T}}_2$:
\begin{align*}
  \mathcal{N}_{00}^2={}&
  \smatr{1&P\\ 0&\R\\ P&\R}\cup \smatr{1&0\\ 0&\R\\ P&N}\cup
  \smatr{0&N\\ 0&0\\ 1&N}\cup \smatr{1&|\\ 0&\mathcal{P}_2\\ 0&|}
  \subseteq\mathbb{P}(\R^3)\times\R^3\,.
\end{align*}
For $d\in\hat{\mathcal{T}}_3$:
\begin{align*}
  \mathcal{N}_{00}^3={}&\smatr{1&P\\ P&\R\\ 0&\R}\cup \smatr{1&0\\ P&N\\0&\R}\cup
  \smatr{1&N\\ 0&N\\ 0&0}\cup \smatr{0&|\\1&\mathcal{P}_2\\0&|}
  \subseteq\mathbb{P}(\R^3)\times\R^3\,.
\end{align*}
For $d\in\hat{\mathcal{T}}_4$:
\begin{align*}
  \mathcal{N}_{00}^4={}&
  \smatr{1&\R\\ P&\R\\ P&\R}\cup \smatr{0&\R\\ 1&P\\ P&\R}\cup
  \smatr{0&\R\\ 1&0\\ P&N}\cup \smatr{1&P\\ 0&\R\\ P&\R}\cup
  \smatr{1&0\\ 0&\R\\ P&N}\cup \smatr{1&P\\ P&\R\\ 0&\R}\cup
  \smatr{1&0\\ P&N\\ 0&\R}\cup \\ &\smatr{1&|\\ 0&\mathcal{P}_2\\ 0&|} \cup
  \smatr{0&|\\ 1&\mathcal{P}_2\\ 0&|}\cup
  \smatr{0&|\\ 0&\mathcal{P}_2\\ 1&|}\subseteq\mathbb{P}(\R^3)\times\R^3\,.
\end{align*}

\subsection{The case $(c,b,D_d,\b)\in\mathscr{B}_{01}$}

For $d\in\hat{\mathcal{T}}_1$:
\begin{align*}
  \mathcal{N}_{01}^1={}&
  \smatr{P&\R\\ 0&N\\ 0&0}\cup \smatr{0&N\\ 0&0\\ 0&0}\subseteq\R^3\times\R^3\,.
\end{align*}
For $d\in\hat{\mathcal{T}}_2$:
\begin{align*}
  \mathcal{N}_{01}^2={}&
\smatr{P&\R\\ 0&\R\\ P&\R}\cup \smatr{0&N\\ 0&0\\ P&\R}\cup
  \smatr{P&\R\\ 0&P\\ 0&\R}\cup \smatr{P&\R\\ 0&0\\ 0&N}\cup
  \smatr{0&N\\ 0&0\\ 0&N}\subseteq\R^3\times\R^3\,.
\end{align*}
For $d\in\hat{\mathcal{T}}_3$:
\begin{align*}
  \mathcal{N}_{01}^3={}&
\smatr{P&\R\\ P&\R\\ 0&\R}\cup \smatr{P&\R\\ 0&N\\ 0&0}\cup
  \smatr{0&P\\ P&\R\\ 0&\R}\cup \smatr{0&0\\ P&\R\\ 0&N}\cup
  \smatr{0&N\\ 0&N\\ 0&0}\subseteq\R^3\times\R^3\,.
\end{align*}
For $d\in\hat{\mathcal{T}}_4$:
\begin{align*}
  \mathcal{N}_{01}^4={}&
\smatr{P&\R\\ P&\R\\ \R&\R}\cup \smatr{P&\R\\ 0&\R\\ P&\R}\cup
  \smatr{0&\R\\ P&\R\\ P&\R}\cup \smatr{P&\R\\ 0&P\\ 0&\R}\cup
  \smatr{P&\R\\ 0&0\\ 0&N}\cup \smatr{0&P\\ P&\R\\ 0&\R}\cup
  \smatr{0&0\\ P&\R\\ 0&N}\cup \\
  &\smatr{0&P\\ 0&\R\\ P&\R}\cup
  \smatr{0&0\\ 0&N\\ P&\R}\cup \smatr{0&|\\ 0&\mathcal{P}_2\\0&|}\subseteq\R^3\times\R^3\,.
\end{align*}

\subsection{The case $(u,c,b,D_d,\b)\in\mathscr{B}_{10}$}

For $d\in\hat{\mathcal{T}}_1$:
\begin{align*}
  \mathcal{N}_{10}^1={}&
\smatr{1&P&P\\ 0&1&\R\\ 0&0&\R}\cup \smatr{1&P&0\\ 0&1&N\\ 0&0&\R}\cup
  \smatr{1&0&|\\ 0&1&\mathcal{P}_2\\ 0&0&|}\cup \smatr{1&1&N\\ 0&0&N\\ 0&0&0}
  \subseteq\mathbb{P}(\R^3)\times\mathbb{P}(\R^3)\times\R^3\,.
\end{align*}
For $d\in\hat{\mathcal{T}}_2$:
\begin{align*}
  \mathcal{N}_{10}^2={}&
  \smatr{1&P&\R\\ 0&1&\R\\ P&\R&\R}\cup \smatr{1&0&\R\\ 0&1&\R\\ P&P&\R}\cup
  \smatr{1&0&P\\ 0&1&\R\\ P&0&\R}\cup \smatr{1&0&0\\ 0&1&\R\\ P&0&N}\cup
  \smatr{1&1&P\\ 0&0&\R\\ P&\R&\R}\cup \smatr{1&1&0\\ 0&0&\R\\ P&\R&N}\cup \\
  &\smatr{1&0&P\\ 0&0&\R\\ P&1&\R}\cup \smatr{1&0&0\\ 0&0&\R\\ P&1&N}\cup \\
  &\left(\left\{\smatr{0\\0\\1}\right\}\times\mathcal{N}_{00}^2 \right) \cup
  \left(\left\{\smatr{0\\0\\1}\right\}\times\mathcal{N}_{00}^4 \right)
  \subseteq\mathbb{P}(\R^3)\times\mathbb{P}(\R^3)\times\R^3\,.
\end{align*}

For $d\in\hat{\mathcal{T}}_3$:
\begin{align*}
  \mathcal{N}_{10}^3={}&
\smatr{1&\R&\R\\ P&1&\R\\ 0&P&\R}\cup \smatr{1&P&\R\\ P&0&\R\\ 0&1&\R}\cup
  \smatr{1&0&P\\ P&0&\R\\ 0&1&\R}\cup \smatr{1&0&0\\ P&0&N\\ 0&1&\R}\cup
  \smatr{1&\R&P\\ P&1&\R\\ 0&0&\R}\cup \smatr{1&\R&0\\ P&1&N\\ 0&0&\R}\cup \\
  &\smatr{1&1&P\\ P&0&\R\\ 0&0&\R}\cup \smatr{1&1&0\\ P&0&N\\ 0&0&\R}\cup \\
  &\left(\left\{\smatr{1\\0\\0}\right\}\times \mathcal{N}_{00}^3\right) \cup
  \left(\left\{\smatr{0\\1\\0}\right\}\times \mathcal{N}_{00}^4\right)
  \subseteq\mathbb{P}(\R^3)\times\mathbb{P}(\R^3)\times\R^3\,.
\end{align*}
For $d\in\hat{\mathcal{T}}_4$:
\begin{align*}
  \mathcal{N}_{10}^4={}&
\smatr{1&|&\R\\ P&\mathcal{P}_1&\R\\ P&|&\R}\cup
  \smatr{0&1&\R\\ 1&P&\R\\ P&\R&\R}\cup \smatr{0&1&\R\\ 1&0&\R\\ P&P&\R}\cup
  \smatr{0&1&\R\\ 1&0&P\\ P&0&\R}\cup \smatr{0&1&\R\\ 1&0&0\\ P&0&N}\cup
  \smatr{0&0&\R\\ 1&1&P\\ P&\R&\R}\cup \\
  &\smatr{0&0&\R\\ 1&1&0\\ P&\R&N}\cup
  \smatr{0&0&\R\\ 1&0&P\\ P&1&\R}\cup \smatr{0&0&\R\\ 1&0&0\\ P&1&N}\cup \\
  &\smatr{1&P&\R\\ 0&1&\R\\ P&\R&\R}\cup \smatr{1&0&\R\\ 0&1&\R\\ P&P&\R}\cup
  \smatr{1&0&P\\ 0&1&\R\\ P&0&\R}\cup \smatr{1&0&0\\ 0&1&\R\\ P&0&N}\cup
  \smatr{1&1&P\\ 0&0&\R\\ P&\R&\R}\cup \smatr{1&1&0\\ 0&0&\R\\ P&\R&N}\cup \\
  &\smatr{1&0&P\\ 0&0&\R\\ P&1&\R}\cup \smatr{1&0&0\\ 0&0&\R\\ P&1&N}\cup \\
  &\smatr{1&P&\R\\ P&\R&\R\\ 0&1&\R}\cup \smatr{1&0&\R\\ P&P&\R\\ 0&1&\R}\cup
  \smatr{1&0&P\\ P&0&\R\\ 0&1&\R}\cup \smatr{1&0&0\\ P&0&N\\ 0&1&\R}\cup
  \smatr{1&1&P\\ P&P&\R\\ 0&0&\R}\cup \smatr{1&1&0\\ P&\R&N\\ 0&0&\R}\cup \\
  &\smatr{1&0&P\\ P&1&\R\\ 0&0&\R}\cup \smatr{1&0&0\\ P&1&N\\ 0&0&\R}\cup \\
  &\left(\left\{\smatr{1\\0\\0},\smatr{0\\1\\0},\smatr{0\\0\\1}\right\}
    \times \mathcal{N}_{00}^4\right)
  \subseteq\mathbb{P}(\R^3)\times\mathbb{P}(\R^3)\times\R^3\,.
\end{align*}

\subsection{The case $(u,c,b,D_d,\b)\in\mathscr{B}_{11}$}

For $d\in\hat{\mathcal{T}}_1$:
\begin{align*}
  \mathcal{N}_{11}^1={}&
\smatr{1&P&\R\\ 0&P&\R\\ 0&0&\R}\cup \smatr{1&0&P\\ 0&P&\R\\ 0&0&\R}\cup
  \smatr{1&0&0\\ 0&P&\R\\ 0&0&N}\cup \smatr{1&P&\R\\ 0&0&N\\ 0&0&0}\cup
  \smatr{1&0&N\\ 0&0&N\\ 0&0&0} \subseteq \mathbb{P}(\R^3)\times\R^3\times\R^3\,.
\end{align*}
For $d\in\hat{\mathcal{T}}_2$:
\begin{align*}
  \mathcal{N}_{11}^2={}&
\smatr{1&P&\R\\ 0&\R&\R\\ P&\R&\R}\cup \smatr{1&0&\R\\0&\R&\R\\P&P&\R}\cup
 \smatr{1&0&P\\ 0&\R&\R\\ P&0&\R}\cup \smatr{1&0&0\\ 0&\R&\R\\ P&0&N}\cup \\
 &\smatr{0&P&\R\\ 0&0&\R\\ 1&P&\R}\cup \smatr{0&P&\R\\ 0&0&P\\ 1&0&\R}\cup
 \smatr{0&P&\R\\ 0&0&0\\ 1&0&N}\cup \smatr{0&0&N\\ 0&0&0\\ 1&P&\R}\cup \\
 &\left(\left\{\smatr{1\\0\\0}\right\}\times \mathcal{N}_{10}^4\right)
   \subseteq\mathbb{P}(\R^3)\times\R^3\times\R^3\,.
\end{align*}
For $d\in\hat{\mathcal{T}}_3$:
\begin{align*}
  \mathcal{N}_{11}^3={}&
\smatr{1&P&\R\\ P&\R&\R\\ 0&\R&\R}\cup \smatr{1&0&\R\\P&P&\R\\0&\R&\R}\cup
  \smatr{1&0&P\\ P&0&\R\\ 0&\R&\R}\cup \smatr{1&0&0\\ P&0&N\\ 0&\R&\R}\cup \\
  &\smatr{1&P&\R\\ 0&P&\R\\ 0&0&\R}\cup \smatr{1&P&\R\\ 0&0&N\\ 0&0&0}\cup
  \smatr{1&0&P\\ 0&P&\R\\ 0&0&\R}\cup \smatr{1&0&0\\ 0&P&\R\\ 0&0&N}\cup \\
  &\left(\left\{\smatr{0\\1\\0}\right\}\times\mathcal{N}_{10}^4\right)
  \subseteq \mathbb{P}(\R^3)\times\R^3\times\R^3\,.
\end{align*}
For $d\in\hat{\mathcal{T}}_4$:
\begin{align*}
  \mathcal{N}_{11}^4={}&
\smatr{1&\R&\R\\ P&\R&\R\\ P&\R&\R}\cup
  \smatr{0&\R&\R\\ 1&P&\R\\ P&\R&\R}\cup \smatr{0&\R&\R\\ 1&0&\R\\ P&P&\R}\cup
  \smatr{0&\R&\R\\ 1&0&P\\ P&0&\R}\cup \smatr{0&\R&\R\\ 1&0&0\\ P&0&N}\cup \\
  &\smatr{1&P&\R\\ 0&\R&\R\\ P&\R&\R}\cup \smatr{1&0&\R\\ 0&\R&\R\\ P&P&\R}\cup
  \smatr{1&0&P\\ 0&\R&\R\\ P&0&\R}\cup \smatr{1&0&0\\ 0&\R&\R\\ P&0&N}\cup \\
  &\smatr{1&P&\R\\ P&\R&\R\\ 0&\R&\R}\cup \smatr{1&0&\R\\ P&P&\R\\ 0&\R&\R}\cup
  \smatr{1&0&P\\ P&0&\R\\ 0&\R&\R}\cup \smatr{1&0&0\\ P&0&N\\ 0&\R&\R}\cup \\
  &\left(\left\{\smatr{1\\0\\0},\smatr{0\\1\\0},\smatr{0\\0\\1}\right\}\times
    \mathcal{N}_{10}^4\right) \subseteq\mathbb{P}(\R^3)\times\R^3\times\R^3\,.
\end{align*}

\bibliographystyle{plain}
\bibliography{invqgi}

\end{document}